\pdfoutput=1
\RequirePackage{ifpdf}
\ifpdf 
\documentclass[pdftex]{sigma}
\else
\documentclass{sigma}
\fi

\numberwithin{equation}{section}
\newtheorem{Theorem}{Theorem}[section]
\newtheorem{Lemma}[Theorem]{Lemma}
\newtheorem{Proposition}[Theorem]{Proposition}
 { \theoremstyle{definition}
\newtheorem{Definition}[Theorem]{Definition} }

\def\R{\operatorname{R}}
\def\RR{\R}
\def\h{{\mathbf{h}}}
\def\Diff{\operatorname{Dif\/f}_{\h}}
\def\Diffs{\operatorname{Dif\/f}_{\h,\sigma}}
\def\der{\bar{\partial}}
\def\Z{x}
\def\th{\tilde{h}}
\def\q{\check{\operatorname{q}}}
\def\U{\operatorname{U}}

\def\gln{{\mathbf{gl}}_n}
\def\V{\operatorname{V}}
\def\F{\operatorname{F}}
\def\PPsi{\Psi}
\def\QQ{\operatorname{Q}}

\begin{document}


\newcommand{\arXivNumber}{1704.05330}

\renewcommand{\thefootnote}{}

\renewcommand{\PaperNumber}{082}

\FirstPageHeading

\ShortArticleName{Dif\/ferential Calculus on $\h$-Deformed Spaces}

\ArticleName{Dif\/ferential Calculus on h-Deformed Spaces\footnote{This paper is a~contribution to the Special Issue on Recent Advances in Quantum Integrable Systems. The full collection is available at \href{http://www.emis.de/journals/SIGMA/RAQIS2016.html}{http://www.emis.de/journals/SIGMA/RAQIS2016.html}}}

\Author{Basile HERLEMONT~$^\dag$ and Oleg OGIEVETSKY~$^{\dag\ddag\S}$}

\AuthorNameForHeading{B.~Herlemont and O.~Ogievetsky}

\Address{$^\dag$~Aix Marseille Univ, Universit\'e de Toulon, CNRS, CPT, Marseille, France}
\EmailD{\href{mailto:herlemont.basile@hotmail.fr}{herlemont.basile@hotmail.fr}, \href{mailto:oleg@cpt.univ-mrs.fr}{oleg@cpt.univ-mrs.fr}}

\Address{$^\ddag$~Kazan Federal University, Kremlevskaya 17, Kazan 420008, Russia}

\Address{$^\S$~On leave of absence from P.N.~Lebedev Physical Institute,\\
\hphantom{$^\S$}~Leninsky Pr.~53, 117924 Moscow, Russia}

\ArticleDates{Received April 18, 2017, in f\/inal form October 17, 2017; Published online October 24, 2017}

\Abstract{We construct the rings of generalized dif\/ferential operators on the ${\bf h}$-deformed vector space of ${\bf gl}$-type. In contrast to the $q$-deformed vector space, where the ring of dif\/ferential operators is unique up to an isomorphism, the general ring of ${\bf h}$-deformed dif\/ferential operators $\operatorname{Dif\/f}_{{\bf h},\sigma}(n)$ is labeled by a rational function $\sigma$ in $n$ variables, satisfying an over-determined system of f\/inite-dif\/ference equations. We obtain the general solution of the system and describe some properties of the rings $\operatorname{Dif\/f}_{{\bf h},\sigma}(n)$.}


\Classification{16S30; 16S32; 16T25; 13B30; 17B10; 39A14}

\renewcommand{\thefootnote}{\arabic{footnote}}
\setcounter{footnote}{0}

\section{Introduction}
As the coordinate rings of $q$-deformed vector spaces, the coordinate rings of $\h$-deformed vector spaces are def\/ined with the help of a solution of the dynamical Yang--Baxter equation. The coordinate rings of $\h$-deformed vector spaces appeared in several contexts. In~\cite{BF} it was observed that such coordinate rings generate the Clebsch--Gordan coef\/f\/icients for ${\rm GL}(2)$. These coordinate rings appear in the study of the cotangent bundle to a quantum group~\cite{AF} and in the study of zero-modes in the WZNW model \cite{AF,FHIOPT,HIOPT}.

The coordinate rings of $\h$-deformed vector spaces appear naturally in the theory of reduction algebras. The reduction algebras \cite{KO1,M,T,Zh} are designed to study the decompositions of representations of an associative algebra $\mathcal{B}$ with respect to its subalgebra $\mathcal{B}'$. Let $\mathcal{B}'$ be the universal enveloping algebra of a reductive Lie algebra $\mathbf{g}$. Let $M$ be a $\mathbf{g}$-module and $\mathcal{B}$ the universal enveloping algebra of the semi-direct product of $\mathbf{g}$ with the abelian Lie algebra formed by $N$ copies of $M$. Then the corresponding reduction algebra is precisely the coordinate ring of~$N$ copies of $\h$-deformed vector spaces.

We restrict our attention to the case $\mathbf{g}=\mathbf{gl}(n)$. Let $V$ be the tautological $\mathbf{gl}(n)$-module and $V^*$ its dual. We denote by $V(n,N)$ the reduction algebra related to~$N$ copies of~$V$ and by~$V^*(n,N)$ the reduction algebra related to~$N$ copies of~$V^*$.

In this article we develop the dif\/ferential calculus on the $\h$-deformed vector spaces of $\mathbf{gl}$-type as it is done in~\cite{WZ} for the
$q$-deformed spaces. Formulated dif\/ferently, we study the consistent, in the sense, explained in Section~\ref{maquse}, pairings between the rings $V(n,N)$ and $V^*(n,N')$. A~consistent pairing allows to construct a f\/lat deformation of the reduction algebra, related to~$N$ copies of $V$ and $N'$ copies of $V^*$. We show that for $N>1$ or $N'>1$ the pairing is essentially unique. However it turns out that for $N=N'=1$ the result is surprisingly dif\/ferent from that for $q$-deformed vector spaces. The consistency leads to an over-determined system of f\/inite-dif\/ference equations for a certain rational function $\sigma$, which we call ``potential'', in~$n$ variables. The solution space $\mathcal{W}$ can be described as follows. Let $\mathbb{K}$ be the ground ring of characteristic~0 and~$\mathbb{K}[t]$ the space of univariate polynomials over $\mathbb{K}$. Then $\mathcal{W}$ is isomorphic to $\mathbb{K}[t]^n$ modulo the $(n-1)$-dimensional subspace spanned by $n$-tuples $(t^j,\dots,t^j)$ for $j=0,1,\dots,n-2$. Thus for each $\sigma\in\mathcal{W}$ we have a ring~$\Diffs(n)$ of generalized $\h$-deformed dif\/ferential operators. The polynomial solutions~$\sigma$ are linear combinations of complete symmetric polynomials; they correspond to the diagonal of~$\mathbb{K}[t]^n$.
The ring $\Diffs(n)$ admits the action of the so-called Zhelobenko automorphisms if and only if the potential $\sigma$ is polynomial.

In Section \ref{corivvz} we give the def\/inition of the coordinate rings of $\h$-deformed vector spaces of $\mathbf{gl}$-type.

Section \ref{hdiff-sect} starts with the description of two dif\/ferent known pairings between $\h$-deformed vector spaces, that is, two dif\/ferent f\/lat deformations of the reduction algebra related to $V\oplus V^*$. The f\/irst deformation is the ring $\Diff(n)$ which is the reduction algebra, with respect to $\gln$, of the classical ring of polynomial dif\/ferential operators. The second ring is related to the reduction algebra, with respect to $\gln$, of the algebra $\U(\mathbf{gl}_{n+1})$. These two examples motivate our study. Then, in Section~\ref{hdiff-sect}, we formulate the main question and results. We present the system of the f\/inite-dif\/ference equations resulting from the Poincar\'e--Birkhof\/f--Witt property of the ring of generalized $\h$-deformed dif\/ferential operators. We obtain the general solution of the system and establish the existence of the potential. We give a characterization of polynomial potentials. We describe the centers of the rings $\Diffs(n)$ and construct an isomorphism between a certain ring of fractions of the ring~$\Diffs(n)$ and a certain ring of fractions of the Weyl algebra. We describe a family of the lowest weight representations and calculate the values of central elements on them. We establish the uniqueness of the deformation in the situation when we have several copies of~$V$ or~$V^*$.

Section \ref{sepro} contains the proofs of the statements from Section~\ref{hdiff-sect}.

{\bf Notation.} We denote by $\mathbb{S}_n$ the symmetric group on $n$ letters. The symbol $s_i$ stands for the transposition $(i,i+1)$.

Let $\h(n)$ be the abelian Lie algebra with generators $\th_i$, $i=1,\dots,n$, and $\U(n)$ its universal enveloping algebra. Set $\th_{ij}=\th_i-\th_j\in \h(n)$. We def\/ine $\bar{\U}(n)$ to be the ring of fractions of the commutative ring $\U(n)$ with respect to the multiplicative set of denominators, generated by the elements $\big(\th_{ij}+a\big)^{-1}$, $a\in\mathbb{Z}$, $i,j=1,\dots,n$, $i\neq j$. Let
\begin{gather}\label{defphichi}
\psi_i:=\prod_{k\colon k>i}\th_{ik}, \psi_i':=\prod_{k\colon k<i}\th_{ik}\qquad \text{and}\qquad \chi_i:=\psi_i\psi_i',\qquad i=1,\dots,n.
\end{gather}
Let $\varepsilon_j$, $j=1,\dots,n$, be the elementary translations of the generators of $\U(n)$, $\varepsilon_j\colon \th_i\mapsto\th_i+\delta_i^j$. For an element $p\in \bar{\U}(n)$ we denote $\varepsilon_j(p)$ by $p[\varepsilon_j]$. We shall use the f\/inite-dif\/ference opera\-tors~$\Delta_j$ def\/ined by
\begin{gather*}\Delta_j f:=f-f[-\varepsilon_j] .\end{gather*}

We denote by $e_L$, $L=0,\dots,n$, the elementary symmetric polynomials in the variables $\th_1,\dots,\th_n$, and by $e(t)$ the generating function of the polynomials~$e_L$,
\begin{gather*}
e_L=\sum_{i_1<\dots<i_L}\th_{i_1}\cdots \th_{i_L} ,\qquad e(t)=\sum_{L=0}^n e_L t^L=\prod_{i=1}^n \big(1+\th_i t\big).\end{gather*}

We denote by $\RR\in\operatorname{End}_{\bar{\U}(n)}\big( \bar{\U}(n)^n\otimes_{\bar{\U}(n)}\bar{\U}(n)^n\big)$ the standard solution of the dynamical Yang--Baxter equation
\begin{gather*} \sum_{a,b,u}{\RR}^{ij}_{ab}\RR^{bk}_{ur}[-\varepsilon_a]\RR^{au}_{mn}=\sum_{a,b,u}
\RR^{jk}_{ab}[-\varepsilon_i]\RR^{ia}_{mu} \RR^{ub}_{nr}[-\varepsilon_m] \end{gather*}
of type~A. The nonzero components of the operator $\RR$ are
\begin{gather}\label{dynRcompb}
\RR_{ij}^{ij}=\frac{1}{\th_{ij}},\qquad i\not=j,\qquad\text{and}\qquad \RR_{ji}^{ij}=
\begin{cases}
 \dfrac{\th_{ij}^2-1}{\th_{ij}^2},& i<j,\\
 1,& i\geq j.
\end{cases}
\end{gather}
We shall need the following properties of $\RR$:
\begin{gather}\label{weze}\RR^{ij}_{kl}[\varepsilon_i+\varepsilon_j]=\RR^{ij}_{kl},\qquad i,j,k,l=1,\dots,n,\\
\label{iceco}\RR^{ij}_{kl}=0\qquad \text{if}\quad (i,j)\neq (k,l)\ \text{or}\ (l,k),\\
\label{Q5do}\RR^2=\text{Id}.\end{gather}

We denote by $\PPsi\in\operatorname{End}_{\bar{\U}(n)}\big( \bar{\U}(n)^n\otimes_{\bar{\U}(n)}\bar{\U}(n)^n\big)$ the dynamical version of the skew inverse of the operator~$\RR$, def\/ined by
\begin{gather}\label{skewn}
\sum_{k,l} \PPsi_{jl}^{ik} \RR_{nk}^{ml}[\varepsilon_m]=\delta_{n}^i\delta_j^m.
\end{gather}
The nonzero components of the operator $\PPsi$ are, see \cite{KO5},
\begin{gather}\label{explpsi}\PPsi^{ij}_{ij}=\QQ^+_i\QQ^-_j\frac{1}{\th_{ij}+1},\qquad \PPsi^{ij}_{ji}= \begin{cases}
1,& i<j,\\
\displaystyle{\frac{\big(\th_{ij}-1\big)^2}{\th_{ij}\big(\th_{ij}-2\big)} },& i>j,\end{cases}\end{gather}
where
\begin{gather*}
\QQ^{\pm}_i=\frac{\chi_i[\pm\varepsilon_i]}{\chi_i} . \end{gather*}

\section{Coordinate rings of h-deformed vector spaces}\label{corivvz}

Let $\F(n,N)$ be the ring with the generators $\Z^{i\alpha}$, $i=1,\dots,n$, $\alpha=1,\dots,N$, and $\th_i$, $i=1,\dots,n$, with the def\/ining relations
\begin{gather}\label{hdesp1a}\th_i\th_j=\th_j\th_i,\qquad i,j=1,\dots,n,\\
\label{hdesp1nn}\th_i\Z^{j\alpha}=\Z^{j\alpha}\big(\th_i+\delta_i^j\big),\qquad i,j=1,\dots,n,\quad \alpha=1,\dots,N.\end{gather}
We shall say that an element $f\in \F(n,N)$ has an $\h(n)$-weight $\omega\in\h(n)^*$ if
\begin{gather}\label{hdesp1aw}\th_if=f\big(\th_i+\omega\big(\th_i\big)\big),\qquad i=1,\dots,n.\end{gather}
The ring $\U(n)$ is naturally the subring of~$\F(n,N)$. Let $\bar{\F}(n,N):=\bar{\U}(n)\otimes_{\U(n)}\F(n,N)$.
The coordinate ring $\V(n,N)$ of~$N$ copies of the $\h$-deformed vector space is the factor-ring of $\bar{\F}(n,N)$ by the relations
\begin{gather}\label{hdesp2}\Z^{i\alpha}\Z^{j\beta}=\sum_{k,l}\RR^{ij}_{kl}\Z^{k\beta}\Z^{l\alpha},\qquad i,j=1,\dots,n,\quad \alpha,\beta=1,\dots,N.\end{gather}
The ring $\V(n,N)$ is the reduction algebra, with respect to $\gln$, of the semi-direct product of $\gln$ and the abelian Lie algebra $V\oplus V\oplus\cdots\oplus V$ ($N$ times) where $V$ is the (tautological) $n$-dimensional $\gln$-module. According to the general theory of reduction algebras \cite{KO1,KO4, Zh}, $\V(n,N)$ is a free left (or right) $\bar{\U}(n)$-module; the ring $\V(n,N)$ has the following Poincar\'e--Birkhof\/f--Witt property:
\begin{gather} \text{given an arbitrary order on the set $\Z^{i\alpha}$, $i=1,\dots,n$, $\alpha=1,\dots,N$, the set}\nonumber\\
\text{of all ordered monomials in $\Z^{i\alpha}$ is a basis of the left $\bar{\U}(n)$-module $\V(n,N)$.} \label{hdesppbw}\end{gather}
Moreover, if $\{\RR_{ij}^{kl}\}_{i,j,k,l=1}^n$ is an arbitrary array of functions in $\th_i$, $i=1,\dots,n$, then the Poincar\'e--Birkhof\/f--Witt property of the algebra def\/ined by the relations~(\ref{hdesp2}), together with the weight prescriptions~(\ref{hdesp1nn}), implies that~$\RR$ satisf\/ies the dynamical Yang--Baxter equation when $N\geq 3$.

Similarly, let $\F^*(n,N)$ be the ring with the generators $\der_{i\alpha}$, $i=1,\dots,n$, $\alpha=1,\dots,N$, and~$\th_i$, $i=1,\dots,n$, with the def\/ining relations (\ref{hdesp1a}) and
\begin{gather}\label{hdesp3}\th_i\der_{j\alpha}=\der_{j\alpha}\big(\th_i-\delta_i^j\big),\qquad i,j=1,\dots,n,\quad \alpha=1,\dots,N.\end{gather}
Let $\bar{\F}^*(n,N):=\bar{\U}(n)\otimes_{\U(n)}\F^*(n,N)$. The $\h(n)$-weights are def\/ined by the same equation~(\ref{hdesp1aw}).
The coordinate ring $\V^*(n,N)$ of $N$ copies of the ``dual'' $\h$-deformed vector space is the factor-ring of $\bar{\F}^*(n,N)$ by the relations
\begin{gather}\label{hdesp4}\der_{l\alpha}\der_{k\beta}=\sum_{i,j}\der_{j\beta}\der_{i\alpha}\RR^{ij}_{kl},\qquad k,l=1,\dots,n,\quad \alpha,\beta=1,\dots,N.\end{gather}
Again, the ring $\V^*(n,N)$ is the reduction algebra, with respect to $\gln$, of the semi-direct product of $\gln$ and the abelian Lie algebra $V^*\oplus V^*\oplus\cdots\oplus V^*$ ($N$~times) where~$V^*$ is the $\gln$-module, dual to~$V$. The ring $\V^*(n,N)$ is a free left (or right) $\bar{\U}(n)$-module; it has a similar to $\V(n,N)$ Poincar\'e--Birkhof\/f--Witt property:
\begin{gather} \text{given an arbitrary order on the set $\der_{i\alpha}$ , $i=1,\dots,n$, $\alpha=1,\dots,N$, the set}\nonumber\\
\text{of all ordered monomials in $\der_{i\alpha}$ is a basis of the left $\bar{\U}(n)$-module $\V^*(n,N)$.}\label{hdesppbwd}
\end{gather}
Again, the Poincar\'e--Birkhof\/f--Witt property of the algebra def\/ined by the relations (\ref{hdesp4}), together with the weight prescriptions~(\ref{hdesp3}), implies that $\RR$ satisf\/ies the dynamical Yang--Baxter equation when $N\geq 3$.

For $N=1$ we shall write $\V(n)$ and $\V^*(n)$ instead of $\V(n,1)$ and $\V^*(n,1)$.

\section{Generalized rings of h-deformed dif\/ferential operators}\label{hdiff-sect}

\subsection{Two examples}\label{twoexse}
Before presenting the main question we consider two examples.

{\bf 1.} We denote by $\text{W}_{n}$ the algebra of polynomial dif\/ferential operators in $n$ variables. It is the algebra with the generators $X^{j}$, $D_{j}$, $j=1,\dots,n$, and the def\/ining relations
\begin{gather*}X^{i}X^{j}=X^{j}X^{i},\qquad D_{i}D_{j}=D_{j}D_{i},\qquad D_{i}X^{j}=\delta_i^j+X^{j}D_{i},\qquad i,j=1,\dots,n.\end{gather*}
The map, def\/ined on the set $\{ e_{ij}\}_{i,j=1}^n$ of the standard generators of $\gln$ by
\begin{gather*}
e_{ij}\mapsto X^{i}D_{j},
\end{gather*}
extends to a homomorphism $\U(\gln)\to \text{W}_{n}$. The reduction algebra of $\text{W}_{n}\otimes \U(\gln)$ with respect to the diagonal embedding
of $\U(\gln)$ was denoted by~$\Diff(n)$ in~\cite{KO5}. It is generated, over $\bar{\U}(n)$, by the images $\Z^{i}$ and $\partial_{i}$,
$i=1,\dots,n$, of the generators~$X^{i}$ and~$D_{i}$. Let
\begin{gather*}\der_{i}:=\partial_{i} \frac{\psi_i}{\psi_i[-\varepsilon_i]},\end{gather*}
where the elements $\psi_i$ are def\/ined in (\ref{defphichi}). Then
\begin{gather}\label{stdi2}\Z^{i}\der_{j}=\sum_{k,l}\der_{k} \RR_{lj}^{ki} \Z^{l}-\delta_{j}^i\sigma_i^{(\text{Dif\/f})} ,
\end{gather}
where $\sigma_i^{(\text{Dif\/f})}=1$, $i=1,\dots,n$. The $\h(n)$-weights of the generators are given by (\ref{hdesp1nn}) and (\ref{hdesp3}). Moreover, the set of the def\/ining relations, over $\bar{\U}(n)$, for the generators~$\Z^{i}$ and $\der_{i}$, $i=1,\dots,n$, consists of (\ref{hdesp2}), (\ref{hdesp4}) (with $N=1$) and~(\ref{stdi2}) (see \cite[Proposition~3.3]{KO5}).

The algebra $\Diff(n,N)$, formed by $N$ copies of the algebra $\Diff(n)$, was used in \cite{KN} for the study of the representation theory of Yangians, and in \cite{KO5} for the R-matrix description of the diagonal reduction algebra of $\gln$ (we refer to \cite{KO2,KO3} for generalities on the diagonal reduction algebras of ${\bf gl}$ type).

{\bf 2.} Identifying each $n\times n$ matrix $a$ with the larger matrix $\left(\begin{smallmatrix}a&0\\0&0\end{smallmatrix}\right)$ gives an embedding of $\gln$ into ${\bf gl}_{n+1}$. The resulting reduction algebra $\mathcal{R}^{\U(\mathbf{gl}_{n+1})}_{\gln}$, or simply $\mathcal{R}^{{\bf gl}_{n+1}}_{\gln}$, was denoted by ${\rm AZ}_n$ in~\cite{Zh}. It is generated, over $\bar{\U}(n)$, by the elements $x^i$, $y_i$, $i=1,\dots,n$, and $\th_{n+1}=z-(n+1)$, where $x^i$ and $y_i$ are the images of the standard generators $e_{i,n+1}$ and $e_{n+1,i}$ of $\U(\mathbf{gl}_{n+1})$
and $z$ is the image of the standard generator $e_{n+1,n+1}$. Let
\begin{gather*}\der_{i}:=y_{i}\frac{\psi_i}{\psi_i[-\varepsilon_i]},\end{gather*}
where the elements $\psi_i$ are def\/ined in~(\ref{defphichi}) (they depend on $\th_1,\dots,\th_n$ only). The $\h(n)$-weights of the generators are given by (\ref{hdesp1nn}) and (\ref{hdesp3}) while
\begin{gather*}\th_{n+1}\Z^i=\Z^i\big(\th_{n+1}-1\big),\qquad \th_{n+1}\der_i=\der_i\big(\th_{n+1}+1\big),\qquad i=1,\dots,n.\end{gather*}
The set of the remaining def\/ining relations consists of (\ref{hdesp2}), (\ref{hdesp4}) (with $N=1$) and
\begin{gather}\label{stdi2a}\Z^{i}\der_{j}=\sum_{k,l}\der_{k} \RR_{lj}^{ki} \Z^{l}-\delta_{j}^i \sigma_i^{({\rm AZ})} ,\end{gather}
where
\begin{gather*}
\sigma_i^{({\rm AZ})}=-\th_i+\th_{n+1}+1, \qquad i=1,\dots,n.\end{gather*}

The algebra ${\rm AZ}_n$ was used in \cite{H} for the study of Harish-Chandra modules and in~\cite{Zh1} for the construction of the Gelfand--Tsetlin bases~\cite{GT}.

The algebra ${\rm AZ}_n$ has a central element
\begin{gather}\label{stdicee} \th_1+\dots+\th_n+\th_{n+1} .\end{gather}
In the factor-algebra $\overline{{\rm AZ}}_n$ of ${\rm AZ}_n$ by the ideal, generated by the element (\ref{stdicee}), the relation (\ref{stdi2a}) is replaced by
\begin{gather}\label{stdi2bb}\Z^{i}\der_{j}=\sum_{k,l}\der_{k} \RR_{lj}^{ki} \Z^{l}-\delta_{j}^i \sigma_i^{(\overline{{\rm AZ}})} ,
\end{gather}
with
\begin{gather*}
\sigma_i^{(\overline{{\rm AZ}})}=-\th_i-\sum_{k=1}^n \th_k+1,\qquad i=1,\dots,n.\end{gather*}

\subsection{Main question and results}\label{maquare}

\subsubsection{Main question}\label{maquse}
Both rings, $\Diff(n)$ and $\overline{{\rm AZ}}_n$ satisfy the Poincar\'e--Birkhof\/f--Witt property. The only dif\/ference between these rings is in the form of the zero-order terms $\sigma_i^{(\text{Dif\/f})}$ and $\sigma_i^{(\overline{{\rm AZ}})}$ in the cross-commutation relations (\ref{stdi2}) and (\ref{stdi2bb}) (compare to the ring of $q$-dif\/ferential operators \cite{WZ} where the zero-order term is essentially~-- up to redef\/initions~-- unique). It is therefore natural to investigate possible generalizations of the rings $\Diff(n)$ and $\overline{{\rm AZ}}_n$. More precisely, given~$n$ elements $\sigma_1,\dots,\sigma_n$ of~$\bar{\U}(n)$, we let $\Diff(\sigma_1,\dots,\sigma_n)$ be the ring, over~$\bar{\U}(n)$, with the generators~$\Z^{i}$ and~$\der_{i}$, $i=1,\dots,n$, subject to the def\/ining relations (\ref{hdesp2}), (\ref{hdesp4}) (with $N=1$) and the oscillator-like relations
\begin{gather}\label{stdi2age}\Z^{i}\der_{j}=\sum_{k,l}\der_{k} \RR_{lj}^{ki} \Z^{l}-\delta_{j}^i \sigma_i.
\end{gather}
The weight prescriptions for the generators are given by~(\ref{hdesp1nn}) and~(\ref{hdesp3}). The diagonal form of the zero-order term (the Kronecker symbol $\delta_{j}^i$ in the right hand side of~(\ref{stdi2age})) is dictated by the $\h(n)$-weight considerations.

We shall study conditions under which the ring $\Diff(\sigma_1,\dots,\sigma_n)$ satisf\/ies the Poincar\'e--Birkhof\/f--Witt property. More specif\/ically, since the rings $\V(n)$ and $\V^*(n)$ both satisfy the Poincar\'e--Birkhof\/f--Witt property, our aim is to study conditions under which $\Diff(\sigma_1,\dots,\sigma_n)$ is isomorphic, as a~$\bar{\U}(n)$-module, to $\V^*(n)\otimes_{\bar{\U}(n)} \V(n)$.

The assignment
\begin{gather}\label{defist}\deg\big(x^i\big)=\deg\big(\der_i\big)=1,\qquad i=1,\dots,n,\end{gather}
def\/ines the structure of a f\/iltered algebra on $\Diff(\sigma_1,\dots,\sigma_n)$. The associated graded algebra is the homogeneous algebra $\Diff(0,\dots,0)$. This homogeneous algebra has the desired Poincar\'e--Birkhof\/f--Witt property because it is the reduction algebra, with respect to $\gln$, of the semi-direct product of $\gln$ and the abelian Lie algebra $V\oplus V^*$.

The standard argument shows that the ring $\Diff(\sigma_1,\dots,\sigma_n)$ can be viewed as a deformation of the homogeneous ring $\Diff(0,\dots,0)$: for the generating set $\big\{x'^ i,\der_i\big\}$, where $x'^i=\hbar x^i$, all def\/ining relations are the same except~(\ref{stdi2age}) in which $\sigma_i$ gets replaced by $\hbar\sigma_i$; one can consider~$\hbar$ as the deformation parameter. Thus our aim is to study the conditions under which this deformation is f\/lat.

\subsubsection{Poincar\'e--Birkhof\/f--Witt property}
It turns out that the Poincar\'e--Birkhof\/f--Witt property is equivalent to the system of f\/inite-dif\/ference equations for the elements $\sigma_1,\dots,\sigma_n\in\bar{\U}(n)$.

\begin{Proposition}\label{equasigmasb} The ring \emph{$\Diff(\sigma_1,\dots,\sigma_n)$} satisf\/ies the Poincar\'e--Birkhoff--Witt property if and only if the elements $\sigma_1,\dots,\sigma_n\in\bar{\U}(n)$ satisfy the following linear system of finite-difference equations
\begin{gather}\label{eqsigib}\th_{ij} \Delta_j \sigma_i = \sigma_i - \sigma_j,\qquad i,j=1,\dots,n. \end{gather}
\end{Proposition}

We postpone the proof to Section \ref{PBWle1}.

\subsubsection[$\Delta$-system]{$\boldsymbol{\Delta}$-system}
The system (\ref{eqsigib}) is closely related to the following linear system of f\/inite-dif\/ference equations for one element $\sigma\in \bar{\U}(n)$:
\begin{gather}\label{eqsigibfopo2} \Delta_i\Delta_j \big( \th_{ij}\sigma\big) =0,\qquad i,j=1,\dots,n. \end{gather}
We shall call it the ``$\Delta$-system". The $\Delta$-system can be written in the form
\begin{gather*}
\th_{ij} \Delta_j\Delta_i \sigma =\Delta_i \sigma -\Delta_j \sigma,\qquad i,j=1,\dots,n. \end{gather*}
We describe the most general solution of the system (\ref{eqsigibfopo2}).

\begin{Definition}\label{vespmj}
Let $\mathcal{W}_j$, $j=1,\dots,n$, be the vector space of the elements of $\bar{\U}(n)$ of the form
\begin{gather*}\frac{\pi\big(\th_j\big)}{\chi_j}\, \qquad \text{where $\pi\big(\th_j\big)$ is a univariate polynomial in $\th_j$},\end{gather*}
and $\chi_j$ is def\/ined in (\ref{defphichi}). Let $\mathcal{W}$ be the sum of the vector spaces $\mathcal{W}_j$, $j=1,\dots,n$.
\end{Definition}

\begin{Theorem}\label{gensolth} An element $\sigma\in\bar{\U}(n)$ satisfies the system \eqref{eqsigibfopo2} if and only if $\sigma\in \mathcal{W}$.
\end{Theorem}

The proof is in Section \ref{secgez}.

The sum $\sum \mathcal{W}_j$ is not direct.

\begin{Definition}\label{vespcoh} Let $\mathcal{H}$ be the $\mathbb{K}$-vector space formed by linear combinations of the complete symmetric polynomials $H_L$, $L=0,1,2,\dots$, in the variables $\th_1,\dots,\th_n$,
\begin{gather*} H_L=\sum_{i_1\leq \dots\leq i_L}\th_{i_1}\cdots \th_{i_L}.\end{gather*}
\end{Definition}

\begin{Lemma}\label{idesy} \quad
\begin{enumerate}\itemsep=0pt
\item[$(i)$] Let $L\in\mathbb{Z}_{\geq 0}$. We have
\begin{gather}\label{idesy1}\sum_{j=1}^n\frac{\th_j^L}{\chi_j}= \begin{cases}0,&L=0,1,\dots,n-2,\\
H_{L-n+1},&L\geq n-1 .\end{cases}
\end{gather}

\item[$(ii)$] The space $\mathcal{H}$ is a subspace of $\mathcal{W}$. Moreover, an element $\sigma\in\U(n)$ satisfies the system~\eqref{eqsigibfopo2} if and only if $\sigma\in \mathcal{H}$, that is,
\begin{gather}\label{cosyinw}\mathcal{H}=\mathcal{W}\cap \U(n). \end{gather}
The symmetric group $\mathbb{S}_n$ acts on the ring $\bar{\U}(n)$ and on the space $\mathcal{W}$ by permutations of the variables $\th_1,\dots,\th_n$. We have
\begin{gather}\label{cosyinwz}\mathcal{H}=\mathcal{W}^{\mathbb{S}_n}, \end{gather}
where $\mathcal{W}^{\, \mathbb{S}_n}$ denotes the subspace of $\mathbb{S}_n$-invariants in $\mathcal{W}$.

\item[$(iii)$] Select $j\in\{1,\dots,n\}$. Then we have a direct sum decomposition
\begin{gather}\label{cosyinw2}\mathcal{W}=\bigoplus_{k\colon k\neq j}\mathcal{W}_k\oplus\mathcal{H}. \end{gather}
\end{enumerate}
\end{Lemma}
The proof is in Section \ref{secgez}.

Let $t$ be an auxiliary indeterminate. We have a linear map of vector spaces $\mathbb{K}[t]^n\to\mathcal{W}$ def\/ined by
\begin{gather*} (\pi_1,\dots,\pi_n)\mapsto \sum_{j=1}^n \frac{\pi_j\big(\th_j\big)}{\chi_j}.\end{gather*}
It follows from Lemma \ref{idesy} that this map is surjective and its kernel is the vector subspace of~$\mathbb{K}[t]^n$ spanned by $n$-tuples $(t^j,\dots,t^j)$ for $j=0,1,\dots,n-2$. The image of the diagonal in~$\mathbb{K}[t]^n$, formed by $n$-tuples $(\pi,\dots,\pi)$, is the space~$\mathcal{H}$.

\subsubsection{Potential}
We shall give a general solution of the system (\ref{eqsigib}).

\begin{Proposition}\label{propotentialb} Assume that the elements $\sigma_1,\dots,\sigma_n\in\bar{\U}(n)$ satisfy the system~\eqref{eqsigib}. Then there exists an element $\sigma\in\bar{\U}(n)$ such that
\begin{gather*}
\sigma_i = \Delta_i \sigma, \qquad i=1,\dots,n.\end{gather*}
\end{Proposition}

We shall call the element $\sigma$ the ``potential'' and write $\Diffs(n)$ instead of $\Diff(\sigma_1,\dots,\sigma_n)$ if $\sigma_i = \Delta_i \sigma$, $i=1,\dots,n$.

According to Proposition \ref{equasigmasb}, the ring $\Diffs(n)$ satisf\/ies the Poincar\'e--Birkhof\/f--Witt property if\/f the potential $\sigma$ satisf\/ies the $\Delta$-system~(\ref{eqsigibfopo2}).

In Section \ref{secpotpr} we give two proofs of Proposition \ref{propotentialb}. In the f\/irst proof we directly describe the space of solutions of the system~(\ref{eqsigib}). As a by-product of this description we f\/ind that the potential exists and moreover belongs to the space~$\mathcal{W}$.

The second proof uses a partial information contained in the system (\ref{eqsigib}) and establishes only the existence of a potential and does not immediately produce the general solution of the system~(\ref{eqsigib}). Given the existence of a potential, the general solution is then obtained by Theorem~\ref{gensolth}.

Let $\mathcal{H}'$ be the $\mathbb{K}$-vector space formed by linear combinations of the complete symmetric polynomials $H_L$, $L=1,2,\dots$, and let
\begin{gather}\label{cosyinw2n}\mathcal{W}'=\bigoplus_{k\colon k\neq 1}\mathcal{W}_k\oplus\mathcal{H}'. \end{gather}
The potential $\sigma$ is def\/ined up to an additive constant, and it will be sometimes useful to uniquely def\/ine~$\sigma$ by requiring that $\sigma\in\mathcal{W}'$.

\subsubsection{A characterization of polynomial potentials}
The polynomial potentials $\sigma\in\mathcal{W}$ can be characterized in dif\/ferent terms. The rings $\Diff(n)$ and $\overline{{\rm AZ}}_n$ admit the action of Zhelobenko automorphisms $\q_1,\dots,\q_{n-1}$ \cite{KO1, Zh2}. Their action on the generators~$\Z^i$ and~$\der_i$, $i=1,\dots,n$, is given by (see~\cite{KO5})
\begin{gather}
\q_i\big(\Z^i\big)=-\Z^{i+1}\frac{\th_{i,i+1}}{\th_{i,i+1}-1},\qquad \q_i\big(\Z^{i+1}\big)=\Z^{i},\quad \q_i\big(\Z^j\big)=\Z^j,\qquad j\not=i,i+1,\nonumber\\
\q_i(\der_{i})=- \frac{\th_{i,i+1}-1}{\th_{i,i+1}}\der_{i+1},\qquad \q_i\big(\der_{i+1}\big)=\der_{i},\qquad \q_i\big(\der_j\big)=\der_j,\qquad
j\not=i,i+1,\nonumber\\
\q_i\big(\th_j\big)=\th_{s_i(j)}.\label{zheautomq}
\end{gather}

\begin{Lemma}\label{posoazhe} The ring \emph{$\Diffs(n)$} admits the action of Zhelobenko automorphisms if and only if~$\sigma$ is a polynomial,
\begin{gather*}\sigma\in\mathcal{H}.\end{gather*}
\end{Lemma}
The proof is in Section \ref{prole5}.

In the examples discussed in Section \ref{twoexse}, the ring $\Diff(n)$ corresponds to $\sigma=H_1$ and the ring $\overline{{\rm AZ}}_n$ corresponds
to $\sigma=-H_2=-\sum\limits_{i,j\colon i\leq j}\th_i\th_j$,
\begin{gather*}\Delta_i H_2 =\th_i+\sum_{k=1}^n \th_k-1.\end{gather*}
The question of constructing an associative algebra which contains $\U(\gln)$ and whose reduction with respect to $\gln$ is $\Diffs(n)$ for $\sigma=H_k$, $k>2$, will be discussed elsewhere.

\subsubsection{Center}
In \cite{HO} we have described the center of the ring $\Diff(n)$. The center of the ring $\Diffs(n)$, $\sigma\in\mathcal{W}$, admits a similar description. Let
\begin{gather*}
\Gamma_i:=\der_i\Z^i\qquad \text{for}\quad i=1,\dots,n.\end{gather*}
Let
\begin{gather}\label{gefufoce} c(t)=\sum_i\frac{e(t)}{1+\th_{i}t }\Gamma_i-\rho(t)=\sum_{k=1}^n c_k t^{k-1},\end{gather}
where $t$ is an auxiliary variable and $\rho(t)$ a polynomial of degree $n-1$ in $t$ with coef\/f\/icients in~$\bar{\U}(n)$.

\begin{Proposition}\label{lece} \quad
\begin{enumerate}\itemsep=0pt
\item[$(i)$] Let $\sigma\in\mathcal{W}$ and $\sigma_j=\Delta_j\sigma$, $j=1,\dots,n$. The elements $c_1,\dots,c_n$ are central in the ring $\Diffs(n)$ if and only if the polynomial $\rho$ satisfies the system of finite-difference equations
\begin{gather}\label{corho}\Delta_j\rho(t)=\frac{e(t)}{1+\th_j t}\sigma_j.\end{gather}
\item[$(ii)$] For an arbitrary $\sigma\in\mathcal{W}$ the system \eqref{corho} admits a solution. Since the system~\eqref{corho} is linear, it is sufficient to present a solution for an element $\sigma\in\mathcal{W}_k$ for each $k=1,\dots,n$, that is, for
\begin{gather}\label{extysi}\sigma=\frac{A\big(\th_k\big)}{\chi_k},\qquad \text{where}\ A\ \text{is a univariate polynomial}.\end{gather}
The solution of the system \eqref{corho} for the element $\sigma$ of the form~\eqref{extysi} is, up to an additive constant from~$\mathbb{K}$,
\begin{gather*}
\rho(t)= \frac{e(t)}{1+\th_k t}\sigma.\end{gather*}
\item[$(iii)$] The center of the ring $\Diffs(n)$ is isomorphic to the polynomial ring $\mathbb{K}[t_1,\dots,t_n]$; the isomorphism is given by $t_j\mapsto c_j$, $j=1,\dots,n$.
\end{enumerate}
\end{Proposition}

The proof is in Section \ref{ceele}.

\subsubsection{Rings of fractions}
In \cite{HO} we have established an isomorphism between certain rings of fractions of the ring $\Diff(n)$ and the Weyl algebra~$\text{W}_n$. It turns out that when we pass to the analogous ring of fractions of the ring $\Diffs(n)$, we loose the information about the potential~$\sigma$. Thus we obtain the isomorphism with the same, as for the ring $\Diff(n)$, ring of fractions of the Weyl algebra~$\text{W}_n$. We denote, as for the ring $\Diff(n)$, by $\text{S}_\Z^{-1}\Diffs(n)$ the localization of the ring $\Diffs(n)$ with respect to the multiplicative set $\text{S}_\Z$ generated by $\Z^j$, $j=1,\dots,n$.

\begin{Lemma}\label{leis2} Let $\sigma$ and $\sigma'$ be two elements of the space $\mathcal{W}'$, see~\eqref{cosyinw2n}.
\begin{enumerate}\itemsep=0pt
\item[$(i)$] The rings $\text{\rm S}_\Z^{-1}\Diffs(n)$ and $\text{\rm S}_\Z^{-1}\text{\rm Dif\/f}_{\h,\sigma'}(n)$ are isomorphic.
\item[$(ii)$] However, the rings $\Diffs(n)$ and $\text{\rm Dif\/f}_{\h,\sigma'}(n)$ are isomorphic, as filtered rings over $\bar{\U}(n)$ $($where the filtration is defined by \eqref{defist}$)$, if and only if
\begin{gather*} \sigma=\gamma\sigma'\qquad \text{for some}\quad \gamma\in\mathbb{K}^*.\end{gather*}
\end{enumerate}
\end{Lemma}

The proof is in Section \ref{seisrf}.

\subsubsection{Lowest weight representations}
The ring $\Diffs(n)$ has an $n$-parametric family of lowest weight representations, similar to the lowest weight representations of the ring $\Diff(n)$, see \cite{HO}. We recall the def\/inition. Let $\mathfrak{D}_n$ be an $\bar{\U}(n)$-subring of $\Diffs(n)$ generated by $\{\der_i\}_{i=1}^n$. Let $\vec{\lambda}:=\{\lambda_1,\dots,\lambda_n\}$ be a sequence, of length $n$, of complex numbers such that $\lambda_i-\lambda_j \notin\mathbb{Z}$ for all $i,j=1,\dots,n$, $i\neq j$. Denote by $M_{\vec{\lambda}}$ the one-dimensional $\mathbb{K}$-vector space with the basis vector $\vert\;\rangle$. The formulas
\begin{gather}\label{dmoac}\th_i\colon \ \vert\; \rangle\mapsto\lambda_i\vert \;\rangle,\qquad \der_i\colon \ \vert\;\rangle\mapsto 0,\qquad i=1,\dots,n,\end{gather}
def\/ine the $\mathfrak{D}_n$-module structure on $M_{\vec{\lambda}}$. The lowest weight representation of lowest weight~$\vec{\lambda}$ is the induced representation $\operatorname{Ind}_{\mathfrak{D}_n}^{\Diffs(n)}M_{\vec{\lambda}}$.

We describe the values of the central polynomial $c(t)$, see (\ref{gefufoce}), on the lowest weight representations.

\begin{Proposition}\label{anosege} The element $c(t)$ acts on $\operatorname{Ind}_{\mathfrak{D}_n}^{\Diffs(n)}M_{\vec{\lambda}}$ by the multiplication on the scalar
\begin{gather}\label{reoacec}-\rho(t)[-\varepsilon],\qquad \text{where}\ \varepsilon=\varepsilon_1+\dots+\varepsilon_n.\end{gather}
\end{Proposition}

The proof is in Section \ref{ansetv}.

\subsubsection{Several copies}
The coexistence of several copies imposes much more severe restrictions on the f\/latness of the deformation. Namely, let $\mathcal{L}$ be the ring with the generators $\Z^{i\alpha}$, $i=1,\dots,n$, $\alpha=1,\dots,N'$, and $\der_{j\beta}$, $j=1,\dots,n$, $\beta=1,\dots,N$ subject to the following def\/ining relations. The $\h(n)$-weights of the generators are given by~(\ref{hdesp1nn}) and (\ref{hdesp3}). The generators~$\Z^{i\alpha}$ satisfy the relations~(\ref{hdesp2}). The generators $\der_{j\beta}$ satisfy the relations~(\ref{hdesp4}). We impose the general oscillator-like cross-commutation relations, compatible with the $\h(n)$-weights, between the generators~$\Z^{i\alpha}$ and~$\der_{j\beta}$:
\begin{gather*}
\Z^{i\alpha}\der_{j\beta}=\sum_{k,l}\der_{k\beta} \RR_{lj}^{ki} \Z^{l\alpha}-\delta_{j}^i\sigma_{i\alpha\beta} ,\qquad i,j=1,\dots,n,\quad \alpha=1,\dots,N',\quad \beta=1,\dots,N,\end{gather*}
with some $\sigma_{i\alpha\beta}\in\bar{\U}(n)$.

\begin{Lemma}\label{seco} Assume that at least one of the numbers $N$ and $N'$ is bigger than~$1$. Then the ring $\mathcal{L}$ has the
Poincar\'e--Birkhoff--Witt property if and only if
\begin{gather}\label{posinb1}\sigma_{i\alpha\beta}=\sigma_{\alpha\beta}\qquad \text{for some}\quad \sigma_{\alpha\beta}\in\mathbb{K}.\end{gather}
\end{Lemma}

The proof is in Section \ref{seseco}.

Making the redef\/initions of the generators, $\Z^{i\alpha}\rightsquigarrow A^{\alpha}_{\alpha'}\Z^{i\alpha'}$ and $\der_{i\beta}\rightsquigarrow B_{\beta}^{\beta'}\der_{i\beta'}$ with some $A\in {\rm GL}(N',\mathbb{K})$ and $B\in {\rm GL}(N,\mathbb{K})$ we can transform the matrix~$\sigma_{\alpha\beta}$ to the diagonal form, with the diagonal $(1,\dots,1,0,\dots,0)$. Therefore, the ring $\mathcal{L}$ is formed by several copies of the rings $\Diff(n)$, $\V(n)$ and~$\V^*(n)$.

\section{Proofs of statements in Section \ref{maquare}}\label{sepro}

\subsection{Poincar\'e--Birkhof\/f--Witt property. Proof of Proposition \ref{equasigmasb}}\label{PBWle1}

The explicit form of the def\/ining relations for the ring $\Diff(\sigma_1,\dots,\sigma_n)$ is
\begin{gather}\label{hdesp6c}\allowdisplaybreaks
\Z^i \Z^j=\frac{\th_{ij}+1}{\th_{ij}}\Z^j \Z^i,\qquad 1\leq i<j\leq n,\\
\label{hdesp8c}
\der_i \der_j=\frac{\th_{ij}-1}{\th_{ij}} \der_j \der_i,\qquad 1\leq i<j\leq n,\\
\label{hdesp9c}
\Z^i \der_j=\begin{cases}
\der_j \Z^i, & 1\leq i<j\leq n, \\
 \dfrac{\th_{ij}\big(\th_{ij}-2\big)}{\big(\th_{ij}-1\big)^2} \der_j \Z^i, & n\geq i>j\geq 1,
\end{cases}
\\ \label{hdesp10c}
\Z^i \der_i = \sum_j \frac{1}{1-\th_{ij}} \der_j \Z^j - \sigma_i, \qquad i=1,\dots,n.
\end{gather}

{\bf Proof of Proposition \ref{equasigmasb}.} We can consider (\ref{hdesp6c}), (\ref{hdesp8c}) and (\ref{hdesp9c}) as the set of ordering relations and use the diamond lemma~\cite{B, Bo} for the investigation of the Poincar\'e--Birkhof\/f--Witt property. The relations~(\ref{hdesp6c}), (\ref{hdesp8c}) and (\ref{hdesp9c}) are compatible with the $\h(n)$-weights of the generators~$\Z^i$ and $\der_i$, $i=1,\dots,n$, so we have to check the possible ambiguities involving the generators~$\Z^i$ and~$\der_i$, $i=1,\dots,n$, only. The properties (\ref{hdesppbw}) and (\ref{hdesppbwd}) show that the ambiguities of the forms~$\Z\Z\Z$ and~$\der\der\der$ are resolvable. It remains to check the ambiguities
\begin{gather}\label{lefam}\Z^i \der_j \der_k\qquad \text{and}\qquad \Z^j \Z^k \der_i.\end{gather}
It follows from the properties (\ref{hdesppbw}) and (\ref{hdesppbwd}) that the choice of the order for the generators with indices $j$ and $k$ in (\ref{lefam}) is irrelevant. Besides, it can be verif\/ied directly that the ring $\Diff(\sigma_1,\dots,\sigma_n)$, with arbitrary $\sigma_1,\dots,\sigma_n\in\bar{\U}(n)$ admits an involutive anti-automorphism $\epsilon$, def\/ined by
\begin{gather}\label{antiautomb} \epsilon\big(\th_i\big)=\th_i,\qquad \epsilon\big(\der_i\big)=\varphi_i\Z^i,\qquad \epsilon\big(\Z^i\big)=\der_i\varphi_i^{-1},\end{gather}
where
\begin{gather*}
\varphi_i:=\frac{\psi_i}{\psi_i[-\varepsilon_i]}=\prod _{k\colon k>i}\frac{\th_{ik}}{\th_{ik}-1}, \qquad i=1,\dots,n.
\end{gather*} By using the anti-automorphism $\epsilon$ we reduce the check of the ambiguity $\Z^j \Z^k \der_i$ to the check of the ambiguity~$\Z^i \der_j \der_k$.

Since the associated graded algebra with respect to the f\/iltration (\ref{defist}) has the Poincar\'e--Birkhof\/f--Witt property, we have, in the check of the ambiguity~$\Z^i \der_j \der_k$, to track only those ordered terms whose degree is smaller than~3. We use the symbol $u\big|_{\text{l.d.t.}}$ to denote the part of the ordered expression for $u$ containing these lower degree terms.

{\bf Check of the ambiguity $\Z^i \der_j \der_k$.} We calculate, for $i,j,k=1,\dots,n$,
\begin{gather}\label{fiwa}\big( x^{i}\der_{j}\big) \der_{k}\big|_{\text{l.d.t.}}=\left( \sum_{u,v}\RR^{ui}_{vj}[\varepsilon_u]\der_{u}x^{v}-
\delta^i_j\sigma_{i}\right)\der_{k}\big|_{\text{l.d.t.}}=
-\sum_{u}\RR^{ui}_{kj}[\varepsilon_u]\der_{u}\sigma_{k}-\delta^i_j\sigma_{i}\der_{k},\end{gather}
and
\begin{gather}
x^{i}\big( \der_{j} \der_{k}\big)\big|_{\text{l.d.t.}} = x^{i}\sum_{a,b}\RR^{ab}_{kj}
\der_{b}\der_{a}\big|_{\text{l.d.t.}}= \sum_{a,b}\RR^{ab}_{kj}[-\varepsilon_i]\left( \sum_{c,d}\RR^{ci}_{db}[\varepsilon_c]\der_{c}x^{d}-\delta^i_b
\sigma_{i}\right)\der_{a}\big|_{\text{l.d.t.}} \nonumber\\
 \hphantom{x^{i}\big( \der_{j} \der_{k}\big)\big|_{\text{l.d.t.}}}{} =
-\sum_{a,b,c}\RR^{ab}_{kj}[-\varepsilon_i] \RR^{ci}_{ab}[\varepsilon_c]\der_{c}\sigma_{a}-\sum_{a}\RR^{ai}_{kj}[-\varepsilon_i]
\sigma_{i}\der_{a}. \label{sewa} \end{gather}
Comparing the resulting expressions in (\ref{fiwa}) and (\ref{sewa}) and collecting coef\/f\/icients in $\der_u$, we f\/ind the necessary and suf\/f\/icient condition for the resolvability of the ambiguity $\Z^i \der_j \der_k$:
\begin{gather}\label{cocoon}\RR^{ui}_{kj}[\varepsilon_u]\sigma_k[\varepsilon_u]+\delta^i_j\delta^u_k\sigma_i=\sum_{a,b} \RR^{ab}_{kj}[-\varepsilon_i]\RR^{ui}_{ab}[\varepsilon_u]
\sigma_a[\varepsilon_u]+\RR^{ui}_{kj}[-\varepsilon_i]\sigma_i,\\
 i,k,j,u=1,\dots,n.\nonumber\end{gather}
Shifting by $-\varepsilon_u$ and using the property (\ref{weze}) together with the ice condition (\ref{iceco}), we re\-wri\-te~(\ref{cocoon}) in the form
\begin{gather}\label{cocoon2}\RR^{ui}_{kj} (\sigma_k-\sigma_i[-\varepsilon_u] )+\delta^i_j\delta^u_k\sigma_i[-\varepsilon_u]=\sum_{a,b} \RR^{ab}_{kj}\RR^{ui}_{ab}\sigma_a.\end{gather}
For $j=k$ the system (\ref{cocoon2}) contains no equations. For $j\neq k$ we have two cases:
\begin{itemize}\itemsep=0pt
\item $u=j$ and $i=k$. This part of the system (\ref{cocoon2}) reads explicitly (see~(\ref{dynRcompb}))
\begin{gather*}\sigma_k-\sigma_k[-\varepsilon_j]=\frac{1}{\th_{kj}} (\sigma_k-\sigma_j ).\end{gather*}
This is the system (\ref{eqsigib}).
\item $u=k$ and $i=j$. This part of the system (\ref{cocoon2}) reads explicitly
\begin{gather*}\frac{1}{\th_{kj}} (\sigma_k-\sigma_j[-\varepsilon_k] )+\sigma_j[-\varepsilon_k]=\frac{1}{\th_{kj}^2}\sigma_k+\frac{\th_{kj}^2-1}{\th_{kj}^2}\sigma_j,\end{gather*}
which reproduces the same system~(\ref{eqsigib}).
\end{itemize}

\subsection[General solution of the system (\ref{eqsigibfopo2}). Proofs of Theorem \ref{gensolth} and Lemma~\ref{idesy}]{General solution of the system (\ref{eqsigibfopo2}).\\ Proofs of Theorem \ref{gensolth} and Lemma~\ref{idesy}}\label{secgez}

We shall interpret elements of $\bar{\U}(n)$ as rational functions on $\h^*$ with possible poles on hyperplanes $\th_{ij}+a=0$, $a\in\mathbb{Z}$, $i,j=1,\dots,n$, $i\neq j$. Let ${\sf M}$ be a subset of $\{1,\dots,n\}$. The symbol $R_{{\sf M}}\bar{\U}(n)$ denotes the subring of $\bar{\U}(n)$ consisting of functions with no poles on hyperplanes $\th_{ij}+a=0$, $a\in\mathbb{Z}$, $i,j\in {\sf M}$, $j\neq i$. The symbol $N_{{\sf M}}\bar{\U}(n)$ denotes the subring of $\bar{\U}(n)$ consisting of functions which do not depend on variables $\th_i$, $i\in {\sf M}$. We shall say that an element $f\in \bar{\U}(n)$ is regular in $\th_j$ if it has no poles on hyperplanes $\th_{jm}+a=0$, $a\in\mathbb{Z}$, $m=1,\dots,n$, $m\neq j$.

{\bf 1.\ Partial fraction decompositions.} 
We will use partial fraction decompositions of an element $f\in\bar{\U}(n)$ with respect to a variable $\th_j$ for some given~$j$. The partial fraction decomposition of $f$ with respect to $\th_j$ is the expression for $f$ of the form
\begin{gather*}f=\mathcal{P}_{j}(f) +\operatorname{reg}_j(f),\end{gather*}
where the elements $\mathcal{P}_{j}(f)$ and $\operatorname{reg}_j(f)$ have the following meaning. The ``regular'' part $\operatorname{reg}_j(f)$ is an element, regular in $\th_j$. The ``principal'' in $\th_j$ part $\mathcal{P}_{j}(f)$ is
\begin{gather*}\mathcal{P}_{j}(f)=\sum_{k\colon k\neq j}\mathcal{P}_{j;k}(f),\end{gather*}
where
\begin{gather}\label{prijk} \mathcal{P}_{j;k}(f)=\sum_{a\in\mathbb{Z}} \sum_{\nu_a\in\mathbb{Z}_{>0}} \frac{u_{ka\nu_a}}{\big(\th_{jk}-a\big)^{\nu_a}},\end{gather}
with some elements $u_{ka\nu_a}\in N_j\bar{\U}(n)$; the sums are f\/inite.

The fact that the ring $\bar{\U}(n)$ admits partial fraction decompositions (that is, that the ele\-ments~$u_{ka\nu_a}$ and $\operatorname{reg}_j(f)$ belong to $\bar{\U}(n)$) is a consequence of the formula
\begin{gather*}\frac{1}{\big(\th_{jk}-a\big)\big(\th_{jl}-b\big)}=\frac{1}{\big(\th_{kl}+a-b\big)}
\left(\frac{1}{\th_{jk}-a}-\frac{1}{\th_{jl}-b}\right).\end{gather*}

{\bf 2.} Let $\text{D}$ be a domain (a commutative algebra without zero divisors) over~$\mathbb{K}$. Let $f$ be an element of $\text{D}\otimes_{\mathbb{K}}\bar{\U}(n)$. Set
\begin{gather}\label{yeqs}\text{Y}_{ij}(f):=\Delta_i\Delta_j \big(\th_{ij}f\big).\end{gather}

\begin{Lemma}\label{idesyij} If $\text{\rm Y}_{ij}(f)=0$ for some $i$ and~$j$, $i\neq j$, then $f$ can be written in the form
\begin{gather}\label{auleij}f=\frac{A}{\th_{ij}}+B,\end{gather}
with some $A,B\in \text{\rm D}\otimes_{\mathbb{K}}R_{i,j}\bar{\U}(n)$.
\end{Lemma}

\begin{proof} We write $f$ in the form
\begin{gather*}f=\frac{A}{\big(\th_{ij}-a_1\big)^{\nu_1}\big(\th_{ij}-a_2\big)^{\nu_2}\cdots \big(\th_{ij}-a_M\big)^{\nu_M}}+B,\end{gather*}
where $a_1< a_2<\dots < a_M$, $\nu_1,\nu_2,\dots,\nu_M\in\mathbb{Z}_{> 0}$, $A,B\in \text{D}\otimes_{\mathbb{K}}R_{i,j}\bar{\U}(n)$ and the element $A$ is
not divisible by any factor in the denominator. There is nothing to prove if $A=0$. Assume that $A\neq 0$. Then
\begin{gather}
0 = \text{Y}_{ij}(f)= \frac{\th_{ij}A}{\big(\th_{ij}-a_1\big)^{\nu_1}\cdots \big(\th_{ij}-a_M\big)^{\nu_M}}
-\frac{\big(\th_{ij}-1\big)A[-\varepsilon_i]}{\big(\th_{ij}-a_1-1\big)^{\nu_1}\cdots \big(\th_{ij}-a_M-1\big)^{\nu_M}}\!\!\!\!\label{subineq}\\
 \hphantom{0 =}{}- \frac{\big(\th_{ij}+1\big)A[-\varepsilon_j]}{\big(\th_{ij}-a_1+1\big)^{\nu_1}\cdots \big(\th_{ij}-a_M+1\big)^{\nu_M}}
+\frac{\th_{ij}A[-\varepsilon_i-\varepsilon_j]}{\big(\th_{ij}-a_1\big)^{\nu_1}\cdots \big(\th_{ij}-a_M\big)^{\nu_M}}+\text{Y}_{ij}(B).\nonumber \end{gather}
The denominator $\big(\th_{ij}-a_M-1\big)$ appears only in the second term in the right hand side of~(\ref{subineq}). It has therefore to be compensated by $\big(\th_{ij}-1\big)$ in the numerator. Hence the only allowed value of $a_M$ is $a_M=0$ and moreover we have $\nu_M= 1$. Similarly, the denominator $\big(\th_{ij}-a_1+1\big)$ appears only in the third term in the right hand side of~(\ref{subineq}) and has to be compensated by $(\th_{ij}+1)$ in the numerator. Hence the only allowed value of $a_1$ is $a_1=0$ and we have $\nu_1= 1$. The inequalities $a_1< a_2<\dots < a_M$ imply that $M=1$ and
we obtain the form~(\ref{auleij}) of~$f$.
\end{proof}

{\bf 3.} Let $f\in \text{D}\otimes_{\mathbb{K}}\bar{\U}(n)$. We shall analyze the linear system of f\/inite-dif\/ference equations
\begin{gather}\label{equdpo} \text{Y}_{ij}(f)=0\qquad\text{for all}\quad i,j=1,\dots,n, \end{gather}
where $\text{Y}_{ij}$ are def\/ined in~(\ref{yeqs}).

First we prove a preliminary result. We recall Def\/inition~\ref{vespmj} of the vector spaces $\mathcal{W}_i$, $i=1,\dots,n$. We select one of the variables $\th_i$, say, $\th_1$.
\begin{Lemma}\label{idesysolc2} Assume that an element $f\in \text{\rm D}\otimes_{\mathbb{K}}\bar{\U}(n)$ satisfies the system~\eqref{equdpo}. Then
\begin{gather}\label{equdposol2}f=\sum_{j=2}^n F_j+\vartheta,\end{gather}
where $\vartheta\in\text{\rm D}\otimes_{\mathbb{K}}\U(n)$ and
\begin{gather}\label{equdposolF2}F_j=\frac{u_j\big(\th_j\big)}{\chi_j}\in \text{\rm D}\otimes_{\mathbb{K}}\mathcal{W}_j \end{gather}
with some univariate polynomials $u_j\big(\th_j\big)$, $j=2,\dots,n$, with coefficients in $\text{\rm D}$.
\end{Lemma}

\begin{proof}Since $\text{Y}_{1m}(f)=0$, $m=2,\dots,n$, Lemma \ref{idesyij} implies that the partial fraction decomposition of $f$ with respect to $\th_1$ has the form
\begin{gather}\label{wxwx0}f=\sum_{m=2}^n\; \frac{\beta_m}{\th_{m1}}+\vartheta,\end{gather}
where $\beta_m\in \text{D}\otimes_{\mathbb{K}} N_1\bar{\U}(n)$,
$m=2,\dots,n$, and $\vartheta\in \text{D}\big[\th_1\big]\otimes_{\mathbb{K}}N_1\bar{\U}(n)$.
Substituting the expression~(\ref{wxwx0}) for $f$ into the equation $\text{Y}_{1j}(f)=0$, $j=2,\dots,n$, we obtain
\begin{gather}
0 = \text{Y}_{1j}(f)= \Delta_1\Delta_j\left( \sum_{m\colon m\neq1,j}\frac{\th_{1j}\beta_m}{\th_{m1}}-\beta_j+\th_{1j}\vartheta\right)\nonumber\\
\hphantom{0}{} =
\Delta_1\Delta_j\left( \sum_{m\colon m\neq1,j}\frac{\th_{1j}\beta_m}{\th_{m1}}+\th_{1j}\vartheta\right)\label{wxwx1}\\
\hphantom{0}{} = \sum_{m\colon m\neq1,j}\left( \frac{\th_{1j}\beta_m}{\th_{m1}}-\frac{\big(\th_{1j}+1\big)\beta_m[-\varepsilon_j]}{\th_{m1}}-\frac{\big(\th_{1j}-1\big)\beta_m}{\th_{m1}+1}
+\frac{\th_{1j}\beta_m[-\varepsilon_j]}{\th_{m1}+1}\right)+\Delta_1\Delta_j\big(\th_{1j}\vartheta\big).\nonumber
\end{gather}
We used that $\beta_m\in \text{D}\otimes_{\mathbb{K}} N_1\bar{\U}(n)$ in the third and fourth equalities. For any $m\neq 1,j$, the terms containing the denominator $\th_{m1}$ in the expression (\ref{wxwx1}) for $\text{Y}_{1j}(f)$ read
\begin{gather*}\frac{1}{\th_{m1}}\big(\th_{1j}\beta_m-\big(\th_{1j}+1\big)\beta_m[-\varepsilon_j]\big).\end{gather*}
Therefore, $\th_{1j}\beta_m-\big(\th_{1j}+1\big)\beta_m[-\varepsilon_j]$ is divisible, as a polynomial in $\th_1$, by $\th_{m1}$, or, what is the same, the
value of $\th_{1j}\beta_m-\big(\th_{1j}+1\big)\beta_m[-\varepsilon_j]$ at $\th_1=\th_m$ is zero. This means that
\begin{gather*}0=\th_{mj}\beta_m-\big(\th_{mj}+1\big)\beta_m[-\varepsilon_j]=\Delta_j\big(\th_{mj}\beta_m\big).\end{gather*}
Therefore, the element $\th_{mj}\beta_m$ does not depend on $\th_j$ for any $j>1$. We conclude that
\begin{gather*}\beta_m=\frac{u_m\big(\th_m\big)}{\prod\limits_{k\colon k\neq1,m}\th_{mk}} \end{gather*}
with some univariate polynomial $u_m$.

We have proved that the element $f$ has the form (\ref{equdposol2}) where $F_j$, $j=2,\dots,n$, are given by~(\ref{equdposolF2}) and the element $\vartheta$ is regular in $\th_1$.

A direct calculation shows that for any $j=2,\dots,n$, the element $F_j$, given by (\ref{equdposolF2}), is a~solution of the linear system~(\ref{equdpo}). Therefore the regular in~$\th_1$ part $\vartheta$ by itself satisf\/ies the system $\text{Y}_{ij}(\vartheta)=0$. It is left to analyze the regular part $\vartheta$.

We use induction in $n$. For $n=2$, the element $\vartheta$ is, by construction, a polynomial in~$\th_1$ and~$\th_2$. This is the induction base. We shall now prove that $\vartheta$ is a polynomial, with coef\/f\/icients in~$\text{D}$, in all $n$ variables $\th_1,\dots,\th_n$.

For arbitrary $n>2$ we have $\vartheta\in\text{D}\big[\th_1\big]\otimes_{\mathbb{K}}\bar{\U}'(n-1)$ where we have denoted by $\bar{\U}'(n-1)$ the subring $N_1\bar{\U}(n)$ of $\bar{\U}(n)$ consisting of functions not depending on $\th_1$. Since $\text{Y}_{ij}(\vartheta)=0$ for $i,j=2,\dots,n$, we can use the induction hypothesis with $n-1$ variables $\th_2,\dots,\th_n$ over the ring $\text{D}'=\text{D}\big[\th_1\big]$.

We now select the variable $\th_2$. It follows from the induction hypothesis that
\begin{gather}\label{comsympam1}\vartheta=\sum_{m\colon m\neq 1,2}\frac{\gamma_m'(\th_m)}{\prod\limits_{l\colon l\neq 1,m}\th_{ml}}+\vartheta',\end{gather}
where $\gamma_m'\big(\th_m\big)$, $m=3,\dots,n$, are univariate polynomials, with coef\/f\/icients in $\text{D}'$, and the ele\-ment~$\vartheta'$ is a polynomial, with coef\/f\/icients in~$\text{D}'$, in the variables $\th_2,\dots,\th_n$. We rewrite the equality~(\ref{comsympam1}) in the form
\begin{gather}\label{comsympa}\vartheta=\sum_{m\colon m\neq 1,2}\frac{\gamma_m\big(\th_m,\th_1\big)}{\prod\limits_{l\colon l\neq 1,m}\th_{ml}}+\vartheta',\end{gather}
with some polynomials $\gamma_m$, $m=3,\dots,n$, in two variables, with coef\/f\/icients in~$\text{D}$; the element~$\vartheta'$ is a polynomial in all variables $\th_1,\dots,\th_n$ with coef\/f\/icients in $\text{D}$.

The equation $\text{Y}_{12}(\vartheta)=0$ for $\vartheta$ given by (\ref{comsympa}) reads
\begin{gather}
0 = \sum_{m\colon m\neq 1,2}\left( \frac{\th_{12}\gamma_m}{\th_{m2}\prod\limits_{l\colon l\neq1,2,m}\th_{ml}}-
\frac{\big(\th_{12}-1\big)\gamma_m[-\varepsilon_1]}{\th_{m2}\prod\limits_{l\colon l\neq1,2,m}\th_{ml}}
- \frac{\big(\th_{12}+1\big)\gamma_m}{\big(\th_{m2}+1\big)\prod\limits_{l\colon l\neq1,2,m}\th_{ml}}\right.\nonumber \\
\left.\hphantom{0=}{} + \frac{\th_{12}\gamma_m[-\varepsilon_1]}{\big(\th_{m2}+1\big)\prod\limits_{l\colon l\neq1,2,m}\th_{ml}}\right)+\text{Y}_{12}(\vartheta').\label{comsympa2} \end{gather}
The terms containing the denominator $\th_{m2}$ in (\ref{comsympa2}) read
\begin{gather*}\frac{1}{\th_{m2}\prod_{l\colon l\neq1,2,m}\th_{ml}}\big( \th_{12}\gamma_m-\big(\th_{12}-1\big)\gamma_m[-\varepsilon_1]\big).\end{gather*}
Therefore, the expression $\th_{12}\gamma_m-\big(\th_{12}-1\big)\gamma_m[-\varepsilon_1]$ is divisible, as a polynomial in $\th_2$, by $\th_{2m}=\th_2-\th_m$, so
\begin{gather*}0=\th_{1m}\gamma_m-\big(\th_{1m}-1\big)\gamma_m[-\varepsilon_1]=\Delta_1\big(\th_{1m}\gamma_m\big).\end{gather*}
Thus the product $\th_{1m}\gamma_m$, $m=3,\dots,n$, does not depend on $\th_1$. Since $\gamma_m$, $m=3,\dots,n$, is a~polynomial, this implies that $\gamma_m=0$. We conclude that $\vartheta=\vartheta'$ and is therefore a polynomial in all variables $\th_1,\dots,\th_n$.
\end{proof}

{\bf 4.} Now we ref\/ine the assertion of Lemma~\ref{idesysolc2}. We shall, at this stage, obtain the general solution of the system~(\ref{equdpo})
in a form which does not exhibit the symmetry with respect to the permutations of the variables $\th_1,\dots,\th_n$.

We recall Def\/inition~\ref{vespcoh} of the vector space $\mathcal{H}$.

\begin{Lemma}\label{idesysolc} \quad
\begin{enumerate}\itemsep=0pt
\item[$(i)$] The general solution of the linear system \eqref{equdpo} for an element $f\in \text{\rm D}\otimes_{\mathbb{K}}\bar{\U}(n)$ has the form
\begin{gather}\label{equdposol}f=\sum_{j=2}^n F_j+\vartheta,\end{gather}
where $F_j\in \text{\rm D}\otimes_{\mathbb{K}}\mathcal{W}_j$ and
\begin{gather}\label{equdposoH}
\vartheta\in\text{\rm D}\otimes_{\mathbb{K}}\mathcal{H}.\end{gather}
\item[$(ii)$] The elements $F_j$, $j=2,\dots,n$, and $\vartheta$ are uniquely defined.
\end{enumerate}
\end{Lemma}

\begin{proof}
(i) In Lemma \ref{idesysolc2} we have established the decomposition (\ref{equdposol}) with $\vartheta\in\text{D}\otimes_{\mathbb{K}}\U(n)$.
We now prove the assertion (\ref{equdposoH}). We f\/irst study the case $n=2$. Let $p\in \text{D}\big[\th_1,\th_2\big]$ be a polynomial such that $\text{Y}_{12}(p)=0$. Since $\Delta_1\Delta_2\big(\th_{12}p\big)=0$ we have $\Delta_2(\th_{12}p)\in \text{D}\big[\th_2\big]$.

It is a standard fact that the operator $\Delta_2$ is surjective on $\text{D}\big[\th_2\big]$. This can be seen, for example, by noticing that the set
\begin{gather*}\th_2^{\uparrow m}:=\th_2\big(\th_2+1\big)\cdots \big(\th_2+m-1\big),\qquad m\in\mathbb{Z}_{\geq 0} ,\end{gather*}
is a basis of $\text{D}\big[\th_2\big]$ over $\text{D}$, and
\begin{gather*}\Delta_2\big(\th_2^{\uparrow m}\big)=m\th_2^{\uparrow m-1}.\end{gather*}
The surjectivity of $\Delta_2$ implies that $\Delta_2\big(\th_{12}p\big)=\Delta_2\big(w\big(\th_2\big)\big)$ for some polynomial $w\big(\th_2\big)\in \text{D}\big[\th_2\big]$. Then $\Delta_2\big(\th_{12}p-w\big(\th_2\big)\big)=0$ so $\th_{12}p-w\big(\th_2\big)=v\big(\th_1\big)$ for some polynomial $v\big(\th_1\big)\in \text{D}\big[\th_1\big]$. Therefore
\begin{gather*}p=\frac{v\big(\th_1\big)+w\big(\th_2\big)}{\th_{12}}=\frac{v\big(\th_1\big)-v\big(\th_2\big)}{\th_{12}}+\frac{v\big(\th_2\big)+w\big(\th_2\big)}{\th_{12}}.\end{gather*}
Since $p$ is a polynomial we must have $w=-v$. Thus
\begin{gather*}p=\frac{v\big(\th_1\big)-v\big(\th_2\big)}{\th_{12}},\end{gather*}
that is, $p$ is a $\text{D}$-linear combination of complete symmetric polynomials in $\th_1$, $\th_2$.

For arbitrary $n$, our polynomial $\vartheta$ is symmetric since, by the above argument, it is symmetric in every pair $\th_i$, $\th_j$ of variables. Moreover, considered as a polynomial in a pair $\th_i$, $\th_j$, it is a~$\text{D}$-linear combination of complete symmetric polynomials in~$\th_i$, $\th_j$. It is then immediate that~$\vartheta$ is a~$\text{D}$-linear combination of complete symmetric polynomials in $\th_1,\dots,\th_n$.

To f\/inish the proof of the statement that the formula (\ref{equdposol}) gives the general solution of the system~(\ref{equdpo}) it is left to check that the complete symmetric polynomials $H_L$, $L=0,1,\dots$, in the variables $\th_1,\dots,\th_n$ satisfy the system~(\ref{equdpo}). Let~$s$ be an auxiliary variable and
\begin{gather}\label{knukn} H(s)=\sum_{L=0}^{\infty}H_Ls^L=\prod_k\frac{1}{1-s\th_k} \end{gather}
be the generating function of the elements $H_L$, $L=0,1,\dots$ It is suf\/f\/icient to show that the formal power series~(\ref{knukn}) satisf\/ies the system~(\ref{equdpo}). Fix $ i,j\in\{1,\dots,n\}$, $i\neq j$, and let $\zeta_{ij}=\frac{1}{(1-\tilde{h}_i s)(1-\tilde{h}_j s)}$. The element
\begin{gather*}\Delta_i \left( \frac{\tilde{h}_{ij}}{\big(1-\tilde{h}_i s\big)\big(1-\tilde{h}_j s\big)} \right) =\frac{1}{1-\tilde{h}_j s}\left(\frac{\tilde{h}_{ij}}{1-\tilde{h}_i s}-\frac{\tilde{h}_{ij}-1}{1-\big(\tilde{h}_i-1\big) s} \right)\\
\hphantom{\Delta_i \left( \frac{\tilde{h}_{ij}}{\big(1-\tilde{h}_i s\big)\big(1-\tilde{h}_j s\big)} \right)}{} =
\frac{1}{\big(1-\tilde{h}_i \tau\big)\big(1-\big(\tilde{h}_i-1\big) \tau\big)} \end{gather*}
does not depend on $\th_j$ so $\text{Y}_{ij}(\zeta_{ij})=0$. Therefore $\text{Y}_{ij}(H(s))=0$ since the factors other than $\zeta_{ij}$ in the product in the right hand side of~(\ref{knukn}) do not depend on $\th_i$ and $\th_j$.

(ii) Finally, the summands in (\ref{equdposol}) are uniquely def\/ined since~(\ref{equdposol}) is a partial fraction decomposition of the element~$f$ in~$\th_1$.
\end{proof}

{\bf 5.\ Proof of Lemma~\ref{idesy}(i).} Let $t$ be an auxiliary indeterminate. Multiplying by $t^{-L-1}$ and taking sum in $L$, we rewrite~(\ref{idesy1}) in the form
\begin{gather*}\sum_{j=1}^n\frac{1}{t-\th_j} \frac{1}{\chi_j}=\frac{1}{\prod\limits_{j=1}^n \big(t-\th_j\big)}.\end{gather*}
The left hand side is nothing else but the partial fraction decomposition, with respect to $t$, of the product in the right hand side.

{\bf 6.\ Proof of Theorem \ref{gensolth}.} The assertion of the Theorem follows immediately from the decomposition (\ref{equdposol}) in Lemma~\ref{idesysolc}
and the identity~(\ref{idesy1}).

{\bf 7.\ Proof of Lemma \ref{idesy}(ii) and~(iii).} (ii) The formula (\ref{cosyinw}) follows from the uniqueness of the decomposition~(\ref{equdposol}) in Lemma~\ref{idesysolc}.

The element $f$ of the form (\ref{equdposol}) is $\mathbb{S}_n$-invariant if and only if $f\in\mathcal{H}$ and the assertion (\ref{cosyinwz}) follows.

(iii) For $j=1$ formula (\ref{cosyinw2}) is the uniqueness statement of Lemma~\ref{idesysolc}. In the proof of Lemma~\ref{idesysolc} we could have selected any $\th_j$ instead of $\th_1$.

\subsection{System (\ref{eqsigib}) }\label{masy}
We proceed to the study of the system (\ref{eqsigib}), that is, the system of equations
\begin{gather}\label{zsyeq}\text{Z}_{ij}=0,\qquad i,j=1,\dots,n,
\end{gather}
where
\begin{gather*} \text{Z}_{ij}=\th_{ij}\Delta_j\sigma_i-\sigma_i+\sigma_j=-\Delta_j\big( \big(\th_{ji}+1\big)\sigma_i\big)+\sigma_j.\end{gather*}
for the $n$-tuple $\sigma_1,\dots,\sigma_n\in\bar{\U}(n)$.

{\bf 1.} We use the equations $\text{Z}_{1j}$, $j=2,\dots,n$, to express the elements $\sigma_j$, $j=2,\dots,n$, in terms of the element $\sigma_1$:
\begin{gather}\label{sijsi1}\sigma_j=\Delta_j\big( \big(\th_{j1}+1\big)\sigma_1\big)=\th_{j1}\Delta_j(\sigma_1)+\sigma_1.\end{gather}
Substituting the expressions (\ref{sijsi1}) into the equations $\text{Z}_{i1}$, $i=2,\dots,n$,
we f\/ind
\begin{gather*}\th_{i1}\big(\Delta_1\big(\th_{i1}\Delta_i\sigma_1+\sigma_1\big)-\Delta_i\sigma_1\big)=0.\end{gather*}
Simplifying by $\th_{i1}$ we obtain
\begin{gather}\label{eqs1js1}\text{W}_i=0,\qquad i=2,\dots,n,\end{gather}
where
\begin{gather*}
\text{W}_i=\Delta_1\big(\th_{i1}\Delta_i\sigma_1+\sigma_1\big)-\Delta_i\sigma_1
=\Delta_i\big(\Delta_1\big(\big(\th_{i1}+1\big)\sigma_1\big)-\sigma_1\big)\\
\hphantom{\text{W}_i}{} =\Delta_i\big( \th_{i1}\sigma_1-\big(\th_{i1}+2\big)\sigma_1[-\varepsilon_1]\big). \end{gather*}
Substituting the expressions (\ref{sijsi1}) into the equations $\text{Z}_{ij}$, $i,j=2,\dots,n$, we f\/ind
\begin{gather*}\th_{i1}\big( \th_{ij}\Delta_i\Delta_j\sigma_1+\Delta_j\sigma_1-\Delta_i\sigma_1\big)=0.\end{gather*}
Simplifying by $\th_{i1}$, we obtain, with the notation~(\ref{yeqs}),
\begin{gather}\label{eqsijs1}\text{Y}_{ij}(\sigma_1)=0,\qquad i,j=2,\dots,n.\end{gather}

This is our f\/irst conclusion which we formulate in the following lemma.

\begin{Lemma}If $\sigma_1,\dots,\sigma_n\in\bar{\U}(n)$ is a solution of the system~\eqref{zsyeq} then the element $\sigma_1$ sa\-tis\-fies the equations~\eqref{eqs1js1} and~\eqref{eqsijs1}. Conversely, if an element $\sigma_1\in\bar{\U}(n)$ satisfies the equations~\eqref{eqs1js1} and~\eqref{eqsijs1} then we reconstruct a solution of the system~\eqref{zsyeq} with the help of the formulas~\eqref{sijsi1}.
\end{Lemma}

{\bf 2.} We shall now analyze the consequences imposed by the equations (\ref{eqs1js1}) on the partial fraction decomposition of the element $\sigma_1$ with respect to $\th_1$. The full form of the expression $W_i$ reads
\begin{gather}\label{eqs1js1c}W_i=\th_{i1}\sigma_1-\big(\th_{i1}+2\big)\sigma_1[-\varepsilon_1]-\big(\th_{i1}-1\big)\sigma_1[-\varepsilon_i]+
\big(\th_{i1}+1\big)\sigma_1[-\varepsilon_1-\varepsilon_i].\end{gather}
We write the element $\sigma_1$ in the form (keeping the notation of Section \ref{secgez})
\begin{gather}\label{decfsi1} \sigma_1=\frac{A}{\big(\th_{i1}-a_1\big)^{\nu_1}\cdots \big(\th_{i1}-a_L\big)^{\nu_L}},\end{gather}
where $a_1< a_2<\dots < a_M$, $\nu_1,\nu_2,\dots,\nu_M\in\mathbb{Z}_{\geq 0}$ and $A\in R_{1,i}\bar{\U}(n)$ is
not divisible by any factor in the denominator.

Substitute the expression (\ref{decfsi1}) into the equation $W_i=0$. The denominator $\big(\th_{i1}-a_L-1\big)$ is present only in term $\big(\th_{i1}-1\big)\sigma_1[-\varepsilon_i]$ in~(\ref{eqs1js1c}). It has therefore to be compensated by $\big(\th_{i1}-1\big)$. Hence the only allowed value of $a_L$ is $a_L=0$ and we have $\nu_L\leq 1$. Similarly, the denominator $\big(\th_{ij}-a_1+1\big)$ appears only in the term $\big(\th_{i1}+2\big)\sigma_1[-\varepsilon_1]$ in~(\ref{eqs1js1c}). It has to be compensated by~$\big(\th_{i1}+2\big)$. Hence the only allowed value of $a_1$ is $a_1=0$ and we have $\nu_1\leq 1$.

It follows that the partial fraction decomposition of the element $\sigma_1$ with respect to $\th_1$ reads
\begin{gather}\label{proms1}\sigma_1=\sum_{k=2}^n\left( \frac{A_k}{\th_{k1}}+\frac{A_k'}{\th_{k1}+1}\right)+B,\end{gather}
where $A_k$, $A_k'$, $k=2,\dots,n$, do not depend on~$\th_1$ and~$B$ is regular in~$\th_1$.

{\bf 3.} The equations (\ref{eqs1js1}) impose further restrictions on the constituents of the decomposition~(\ref{proms1}) of the element~$\sigma_1$.
Substitute the decomposition (\ref{proms1}) into the equation $W_i=0$. The terms which have denominators of the form $\th_{i1}+m$, $m\in\mathbb{Z}$, in (\ref{eqs1js1c}) are
\begin{gather}
 \th_{i1}\left( \frac{A_i}{\th_{i1}}+\frac{A_i'}{\th_{i1}+1}\right)-\big(\th_{i1}+2\big)\left( \frac{A_i}{\th_{i1}+1}+\frac{A_i'}{\th_{i1}+2}\right)\nonumber \\
\qquad{} - \big(\th_{i1}-1\big)\left( \frac{A_i[-\varepsilon_i]}{\th_{i1}-1}+\frac{A_i'[-\varepsilon_i]}{\th_{i1}}\right)+\big(\th_{i1}+1\big)\left( \frac{A_i[-\varepsilon_i]}{\th_{i1}}+\frac{A_i'[-\varepsilon_i]}{\th_{i1}+1}\right).\label{proms1b} \end{gather}
In the expression (\ref{proms1b}), the terms with the denominator $\th_{i1}+1$ read
\begin{gather*}\frac{\th_{i1}A_i'-\big(\th_{i1}+2\big)A_i+\big(\th_{i1}+1\big)A_i'[-\varepsilon_i]}{\th_{i1}+1}=-\frac{A_i'+A_i}{\th_{i1}+1}+A_i'-A_i+
A_i'[-\varepsilon_i].\end{gather*}
Therefore,
\begin{gather*}A_i+A_i'=0,\qquad i=2,\dots,n.\end{gather*}
With this condition, the expression (\ref{proms1b}) vanishes.

We conclude that
\begin{gather}\label{proms1c}\sigma_1=\sum_{k=2}^n \left( \frac{A_k}{\th_{k1}}-\frac{A_k}{\th_{k1}+1}\right)+B.\end{gather}

{\bf 4.} Now we substitute the obtained expression (\ref{proms1c}) for $\sigma_1$ into the equation $W_j=0$ with $j\neq i$ and follow the singularities of the form $\th_{i1}+m$, $m\in\mathbb{Z}$. The singular terms are
\begin{gather}
\th_{j1}\left( \frac{A_i}{\th_{i1}}-\frac{A_i}{\th_{i1}+1}\right)-\big(\th_{j1}+2\big)\left( \frac{A_i}{\th_{i1}+1}-\frac{A_i}{\th_{i1}+2}\right) \nonumber\\ \qquad{} -\big(\th_{j1}-1\big)\left( \frac{A_i[-\varepsilon_j]}{\th_{i1}}-\frac{A_i[-\varepsilon_j]}{\th_{i1}+1}\right)+(\th_{j1}+1)\left( \frac{A_i[-\varepsilon_j]}{\th_{i1}+1}-\frac{A_i[-\varepsilon_j]}{\th_{i1}+2}\right). \label{proms1cw} \end{gather}
In the expression (\ref{proms1cw}), the terms with the denominator $\th_{i1}$ read
\begin{gather*}\frac{ \th_{j1}A_i-\big(\th_{j1}-1\big)A_i[-\varepsilon_j]}{\th_{i1}}.\end{gather*}
Therefore, the numerator, as a polynomial in $\th_1$, must be divisible by the denominator $\th_{i1}$. The polynomial remainder of this division equals
\begin{gather*}\th_{ij}A_i-\big(\th_{ij}+1\big)A_i[-\varepsilon_j]=\Delta_j\big(\th_{ij}A_i\big).\end{gather*}
Therefore, for any $j=2,\dots,n$, $j\neq i$, the combination $\th_{ij}A_i$ does not depend on $\th_j$. It follows that
\begin{gather*}A_i=\frac{\alpha_i\big(\th_i\big)}{\prod\limits_{l\colon l\neq 1,i}\th_{il}},\qquad i=2,\dots,n,\end{gather*}
where each $\alpha_i$ is a univariate polynomial.

For the moment, we have found that
\begin{gather*} \sigma_1=\sigma_1^{(s)}+B,\end{gather*}
where the element $B$ is regular in $\th_1$ and
\begin{gather*}
\sigma_1^{(s)}=\sum_{i=2}^n\left( \frac{1}{\th_{i1}}-\frac{1}{\th_{i1}+1}\right)\frac{\alpha_i\big(\th_i\big)}{\prod\limits_{l\colon l\neq 1,i}\th_{il}}.\end{gather*}
A direct calculation shows that the element $\sigma_1^{(s)}$ satisf\/ies the equations~(\ref{eqs1js1}) and~(\ref{eqsijs1}), so it is left to analyze the regular in $\th_1$ part $B$.

{\bf 5.} Since the element $B$ satisf\/ies the system of equations~(\ref{eqsijs1}), we can use the results of Lemma~\ref{idesysolc} with $\text{D}=\mathbb{K}[\th_1]$. According to Lemma~\ref{idesysolc}, we can write (with an obvious shift in indices) the partial fraction decomposition of the element $B$ with respect to $\th_2$ in the form
\begin{gather*}
B=\sum_{j=3}^n \frac{u_j\big(\th_j,\th_1\big)}{\prod\limits_{l\colon l\neq 1,j}\th_{jl}}+C,\end{gather*}
where $u_j\big(\th_j,\th_1\big)$, $j=3,\dots,n$, is a polynomial in $\th_j$, $\th_1$ and $C$ is a linear combination of complete symmetric polynomials in $\th_2,\dots,\th_n$ with coef\/f\/icients in $\mathbb{K}[\th_1]$.

The equation $W_2(B)=0$ implies that the expression
\begin{gather*}
\th_{21}B-\big(\th_{21}+2\big)B[-\varepsilon_1]\end{gather*}
does not depend on $\th_2$. In the notation of paragraph~1 in Section~\ref{secgez}, the part $\mathcal{P}_{2;j}$, $j=3,\dots,n$, of this expression is
\begin{gather*}\frac{1}{\prod\limits_{l\colon l\neq 12,j}\th_{jl}}\left(\frac{\th_{21}u_j\big(\th_j,\th_1\big)-\big(\th_{21}+2\big)u_j\big(\th_j,\th_1-1\big)
}{\th_{j2}}\right).\end{gather*}
Therefore, $\th_{21}u_j\big(\th_j,\th_1\big)-\big(\th_{21}+2\big)u_j\big(\th_j,\th_1-1\big)$ is divisible, as polynomial in $\th_2$, by $\th_{2j}$. So the value of $\th_{21}u_j\big(\th_j,\th_1\big)-\big(\th_{21}+2\big)u_j\big(\th_j,\th_1-1\big)$
at $\th_2=\th_j$ is zero,
\begin{gather}\label{proms1e}\th_{j1}u_j\big(\th_j,\th_1\big)-\big(\th_{j1}+2\big)u_j\big(\th_j,\th_1-1\big)=0.\end{gather}
Set
\begin{gather}\label{proms1ex}u_j=\frac{\beta_j}{\th_{j1}\big(\th_{j1}+1\big)}.\end{gather}
Then equation~(\ref{proms1e}) becomes
\begin{gather*}\frac{\beta_j}{\th_{j1}+1}+\frac{\beta_j[-\varepsilon_1]}{\th_{j1}+1}=0,\end{gather*}
or $\Delta_1(\beta_j)=0$, so $\beta_j$ depends only on $\th_j$. But then if $\beta_j\neq 0$, the formula~(\ref{proms1ex}) shows that $u_j$ cannot be a polynomial in~$\th_1$.

We conclude that the principal part of the element $B$ with respect to $\th_2$ vanishes, and $B=C$ is a polynomial in all its variables.

{\bf 6.} We claim that $C$ is a $\mathbb{K}$-linear combination of the elements $\Delta_1(H_L)$, $L=1,2,\dots$, where~$H_L$ are the complete symmetric polynomials in $\th_1,\dots,\th_n$.

Consider f\/irst the case $n=2$. Set
\begin{gather*}C=\frac{\xi}{\th_{21}\big(\th_{21}+1\big)},\end{gather*}
where $\xi$ is some polynomial in $\th_1$ and $\th_2$.
With this substitution the equation $W_2(C)=0$ becomes
\begin{gather*}\Delta_2\left( \frac{1}{\th_{21}+1}\Delta_1(\xi)\right)=0, \end{gather*}
that is,
\begin{gather*}\frac{1}{\th_{21}+1}\Delta_1(\xi)=\mu,\end{gather*}
where $\mu$ does not depend on $\th_2$. Note that by construction, the polynomial $\xi$ is divisible by $\th_{21}\big(\th_{21}+1\big)$, which implies that $\mu$ is a polynomial in $\th_1$. Since $\Delta_1$ is surjective on polynomials, we can write $\mu=\Delta_1^2 \big(z\big(\th_1\big)\big)$ for some univariate polynomial $z$, that is
\begin{gather*}\Delta_1(\xi)=\big(\th_{21}+1\big)\Delta_1^2\big(z\big(\th_1\big)\big).\end{gather*}
We have
\begin{gather*}\big(\th_{21}+1\big)\Delta_1^2\big(z\big(\th_1\big)\big)=\Delta_1\big(\th_{21}\Delta_1\big(z\big(\th_1\big)\big)+z\big(\th_1\big) \big).\end{gather*}
Therefore,
\begin{gather*}\Delta_1\big( \xi-\th_{21}\Delta_1\big(z\big(\th_1\big)\big)-z\big(\th_1\big)\big)=0,\qquad \text{or}\qquad \xi=\th_{21}\Delta_1\big(z\big(\th_1\big)\big)+z\big(\th_1\big)+w\big(\th_2\big),\end{gather*}
where $w\big(\th_2\big)$ is a polynomial in $\th_2$. That is,
\begin{gather}\label{cujcuj}C=\frac{\Delta_1\big(z\big(\th_1\big)\big)}{\th_{21}+1}+\frac{z\big(\th_1\big)+w\big(\th_2\big)}{\th_{21}\big(\th_{21}+1\big)}.\end{gather}
Since the element $C$ is a polynomial, the denominator $\th_{21}$ in the second term in the right hand side of~(\ref{cujcuj}) shows that $w=-z$. Therefore,
\begin{gather*}C=\frac{\th_{21}\Delta_1\big(z\big(\th_1\big)\big)+z\big(\th_1\big)-z\big(\th_2\big)}{\th_{21}\big(\th_{21}+1\big)}=\Delta_1\left( \frac{z\big(\th_1\big)-z\big(\th_2\big)}{\th_{21}}\right),\end{gather*}
as claimed.

The claim for arbitrary $n$ follows since for any $j>2$ the element $C$ is a linear combination of $\Delta_1\big(H_L\big(\th_1,\th_j\big)\big)$, $L=1,2,\dots$

{\bf 7.} We summarize the results of this section in the following proposition.

\begin{Proposition} The general solution of the system~\eqref{eqs1js1} and \eqref{eqsijs1} has the form
\begin{gather}\label{proms1dn}\sigma_1=\sum_{i=2}^n\left( \frac{1}{\th_{i1}}-\frac{1}{\th_{i1}+1}\right)\frac{\alpha_i\big(\th_i\big)}{\prod\limits_{l\colon l\neq 1,i}\th_{il}}+\Delta_1(\nu),
\qquad \nu\in\mathcal{H} \end{gather}
and $\alpha_2,\dots,\alpha_n$ are univariate polynomials. The elements $\alpha_2,\dots,\alpha_n$ and $\nu$ are uniquely defined.
\end{Proposition}

\subsection{Potential. Proof of Proposition \ref{propotentialb}}\label{secpotpr}
\paragraph{First proof.} We rewrite the formula (\ref{proms1dn}) in the form
\begin{gather*}\sigma_1=\Delta_1(\sigma),\qquad \text{where}\quad \sigma=\sum_{i=2}^n \frac{\alpha_i\big(\th_i\big)}{\chi_i}+\nu\in\mathcal{W}
.\end{gather*}
Then the expressions for the elements $\sigma_j$, $j=2,\dots,n$, see (\ref{sijsi1}), read
\begin{gather}\label{nuina}\sigma_j=\Delta_j\left( \big(\th_{j1}+1\big)\sum_{i=2}^n \frac{\alpha_i\big(\th_i\big)}{\big(\th_{i1}+1\big)\chi_i}+\big(\th_{j1}+1\big)\Delta_1(\nu)\right).\end{gather}
Since, for $\nu\in\mathcal{H}$,
\begin{gather*}\Delta_1\Delta_j\big(\th_{j1}\nu\big)=0,\end{gather*}
we f\/ind that
\begin{gather*}\Delta_j\big(\big(\th_{j1}+1\big)\Delta_1(\nu)\big)=\Delta_j(\nu).\end{gather*}

The term with $i=j$ in the sum in the right hand side of (\ref{nuina}) is simply
\begin{gather*}\frac{\alpha_j\big(\th_j\big)}{\chi_j}.\end{gather*}
Since
\begin{gather*}\frac{\th_{j1}+1}{\th_{i1}+1}=\frac{\th_{i1}+1+\th_{ji}}{\th_{i1}+1}=1+\frac{\th_{ji}}{\th_{i1}+1},\end{gather*}
we can rewrite the term with $i\neq j$ in the right hand side of (\ref{nuina}) in the form
\begin{gather*}\Delta_j\left( \big(\th_{j1}+1\big)\frac{\alpha_i\big(\th_i\big)}{\big(\th_{i1}+1\big)\chi_i}\right)=\Delta_j\left( \frac{\alpha_i\big(\th_i\big)}{\chi_i}-
\frac{\alpha_i\big(\th_i\big)}{\big(\th_{i1}+1\big)\prod\limits_{l\neq i,j}\th_{il}}\right)=\Delta_j\left( \frac{\alpha_i\big(\th_i\big)}{\chi_i} \right).\end{gather*}
Therefore,
\begin{gather*}\sigma_j=\Delta_j(\sigma)\qquad \text{for all}\quad j=1,\dots, n.\end{gather*}
The proof of Proposition \ref{propotentialb} is completed.

{\bf Second proof.} Let $p\in \U(n)$ be a polynomial such that $\Delta_1\Delta_2 (p)=0$. Thus, $\Delta_2(p)$ does not depend on $\th_1$ so, by surjectivity of $\Delta_2$ on polynomials in~$\th_2$, there exists a polynomial $p_1$ which does not depend on $\th_1$ and $\Delta_2(p)=\Delta_2(p_1)$. The polynomial $p_2:=p-p_1$ does not depend on~$\th_2$. The next lemma generalizes this decomposition
\begin{gather}\label{sepr1a}p=p_1+p_2,\qquad \Delta_1(p_1)=0,\qquad \Delta_2(p_2)=0,\end{gather}
to the ring $\bar{\U}(n)$.

\begin{Lemma}\label{diseder} Let $f\in \bar{\U}(n)$. If
\begin{gather*}
\Delta_1\Delta_2 (f)=0 \end{gather*}
then there exist elements $f_1,f_2\in\bar{\U}(n)$ such that $f_1$ does not depend on $\th_1$, $f_2$ does not depend on $\th_2$, and
\begin{gather}\label{sepr2}f=f_1+f_2. \end{gather}
\end{Lemma}

\begin{proof} Decompose $f$ into partial fractions with respect to $\th_1$.

We have $\mathcal{P}_{1;2}(f)=0$. Indeed, write $\mathcal{P}_{1;2}(f)$ in the form
\begin{gather*}\mathcal{P}_{1;2}(f)=\frac{u}{\big(\th_{12}-a_1\big)^{\nu_1}\cdots \big(\th_{12}-a_L\big)^{\nu_L}},\end{gather*}
where $a_1< a_2<\dots < a_L$, $\nu_1,\nu_2,\dots,\nu_L\in\mathbb{Z}_{> 0}$ and $u\in R_{1,2}\bar{\U}(n)$ is not divisible by any factor in the denominator. Assume that $u\neq 0$. Then
\begin{gather*} \Delta_1\Delta_2 (\mathcal{P}_{1;2}(f))=
\frac{u+u[-\varepsilon_1-\varepsilon_2]}{\big(\th_{12}-a_1\big)^{\nu_1}\cdots \big(\th_{12}-a_L\big)^{\nu_L}}- \frac{u[-\varepsilon_1]}{\big(\th_{12}-a_1-1\big)^{\nu_1}\cdots \big(\th_{12}-a_L-1\big)^{\nu_L}} \\
\hphantom{\Delta_1\Delta_2 (\mathcal{P}_{1;2}(f))= }{}-\frac{u[-\varepsilon_2]}{\big(\th_{12}-a_1+1\big)^{\nu_1}\cdots \big(\th_{12}-a_L+1\big)^{\nu_L}}. \end{gather*}
The factor $\big(\th_{12}-a_L-1\big)$ appears only in the denominator of the second term in the right hand side and cannot be compensated by the numerator. Thus $\mathcal{P}_{1;2}(f)=0$ (the consideration of the factor $\big(\th_{12}-a_1+1\big)$ in the denominator of the third term proves the claim as well).

Now we write the part $\mathcal{P}_{1;j}(f)$, $j>2$, in the form (\ref{prijk}),
\begin{gather*} \mathcal{P}_{1;j}(f)=\sum_{a\in\mathbb{Z}} \sum_{\nu_a\in\mathbb{Z}_{>0}} \frac{u_{ja\nu_a}}{\big(\th_{1j}-a\big)^{\nu_a}},\end{gather*}
where $u_{ja\nu_a}\in N_1\bar{\U}(n)$ and the sums are f\/inite. Then
\begin{gather}\label{xurabsu} \Delta_1\Delta_2 (\mathcal{P}_{1;j}(f))=\sum_{a\in\mathbb{Z}} \sum_{\nu_a\in\mathbb{Z}_{>0}} \Delta_2(u_{ja\nu_a})\left(\frac{1}{\big(\th_{1j}-a\big)^{\nu_a}}
-\frac{1}{\big(\th_{1j}-a-1\big)^{\nu_a}}\right).\end{gather}
We prove that the elements $u_{ja\nu_a}$ do not depend on~$\th_2$. Indeed, if this is not true then there is a minimal $a\in\mathbb{Z}$ for which
$\Delta_2(u_{ja\nu_a})\neq 0$ for some~$\nu_a$. But then the denominator $(\th_{1j}-a)^{\nu_a}$ in the right hand side in~(\ref{xurabsu}) cannot be compensated.

We conclude that $f=f_{2,0}+g$ where $f_{2,0}=\sum\limits_{j>2}\mathcal{P}_{1;j}(f)$ does not depend on $\th_2$ and $g$ is regular in $\th_1$.

We decompose $g$ with respect to $\th_2$. As above, the part $\mathcal{P}_{2;1}(g)$ vanishes and the calculation, parallel to (\ref{xurabsu}), shows that
$\mathcal{P}_{2;j}(g)$, $j>2$, does not depend on $\th_1$. Now we have
\begin{gather*}f=f_{2,0}+f_{1,0}+f^{+},\end{gather*}
where $f_{1,0}=\sum\limits_{j>2}\mathcal{P}_{2;j}(g)$ does not depend on $\th_1$ and $f^{+}$ is regular in $\th_1$ and $\th_2$.

We use the decomposition (\ref{sepr1a}) for the regular part $f^{+}$ and write $f^{+}=f^{+}_1+f^{+}_2$, where $f^{+}_1$ does not depend on $\th_1$ and $f^{+}_2$ does not depend on $\th_2$. This leads to the required decomposition~(\ref{sepr2}) with $f_1=f_{1,0}+f^{+}_1$ and $f_2=f_{2,0}+f^{+}_2$.
\end{proof}

\begin{Lemma}\label{exiopo} Let $\sigma_1,\dots,\sigma_k$, $k\leq n$, be a $k$-tuple of elements in $\bar{\U}(n)$ such that
\begin{gather*}
\Delta_a(\sigma_b)=\Delta_b (\sigma_a),\qquad a,b=1,\dots,k.\end{gather*}
Assume that $\sigma_a$ belongs to the image of $\Delta_a$ for all $a=1,\dots,k$, that is, there exist elements $f_1,\dots,f_k\in \bar{\U}(n)$ for which $\sigma_a=\Delta_a(f_a)$, $a=1,\dots,k$. Then there exists a potential $f\in \bar{\U}(n)$ such that
\begin{gather*}
\sigma_a=\Delta_a (f)=0,\qquad a=1,\dots,k. \end{gather*}
\end{Lemma}

\begin{proof} For $k=1$ there is nothing to prove. Let now $k>1$. We use the induction in $k$. By the induction hypothesis, there exist elements $F,G\in\bar{\U}(n)$ such that
\begin{gather*}\sigma_a=\Delta_a(F)\qquad \text{for}\quad a=1,3,\dots,k\qquad \text{and}\qquad \sigma_b=\Delta_b(G)\qquad \text{for}\quad b=2,3,\dots,k.\end{gather*}
Then
\begin{gather*}\Delta_c(F)=\Delta_c(G) \qquad \text{for}\quad c=3,\dots,k\qquad \text{and}\qquad \Delta_1\Delta_2(G)=\Delta_2\Delta_1(F).\end{gather*}
The element $F-G$ does not depend on $\th_c$, $c=3,\dots,k$, and $\Delta_1\Delta_2(F-G)=0$. According to Lemma~\ref{diseder}, there exist two elements $u,v\in \bar{\U}(n)$ such that $u$ does not depend on $\th_2$, $v$ does not depend on $\th_1$, and $F-G=u-v$. Then
\begin{gather*}f:=F+v=G+u\end{gather*}
is the desired potential. \end{proof}

{\bf Second proof of Proposition \ref{propotentialb}}. The symmetric, in $i$ and $j$, part of the equation (\ref{eqsigib}) is
\begin{gather}\label{foposy}\Delta_i \sigma_j=\Delta_j \sigma_i.\end{gather}
The system (\ref{foposy}) by itself does not imply the existence of a potential. However, the equation~(\ref{eqsigib}) can be written in the form
$\sigma_j=\Delta_j\big( \big(\th_{ji}+1\big)\sigma_i\big)$. So for each $j=1,\dots,n$ the ele\-ment~$\sigma_j$ belongs to the image of the operator $\Delta_j$. Then, according to Lemma~\ref{exiopo}, there exists $\sigma\in\bar{\U}(n)$ such that $\sigma_j=\Delta_j(\sigma)$.

\subsection{Polynomial potentials. Proof of Lemma \ref{posoazhe}}\label{prole5}
The operator $\q_i$ def\/ined by (\ref{zheautomq}) can be an automorphism of the ring $\Diffs(n)$ only if
\begin{gather}\label{klja1}\q_i(\sigma_j)=\sigma_{s_i(j)}=\Delta_{s_i(j)}(\sigma),\qquad i,j=1,\dots,n.\end{gather}
On the other hand,
\begin{gather}\label{klja2}\q_i(\sigma_j)=\q_i(\Delta_j(\sigma))=\Delta_{s_i(j)}(\q_i(\sigma)),\qquad i,j=1,\dots,n.\end{gather}
Comparing (\ref{klja1}) and (\ref{klja2}) we obtain
\begin{gather*}\Delta_j(\sigma-\q_i(\sigma))=0,\qquad i,j=1,\dots,n,\end{gather*}
which implies that $\sigma$ is $\mathbb{S}_n$-invariant. The assertion now follows from Lemma \ref{idesy}(ii).

\subsection{Central elements. Proof of Proposition \ref{lece}}\label{ceele}
(i) To analyze the relation $x^j c(t)-c(t)x^j=0$, we shall write the expression $x^j c(t)-c(t)x^j$ in the ordered form, in the order $\der\Z\Z$.
The element
\begin{gather*}
c_0(t)=\sum_i\frac{e(t)}{1+\th_{i}t }\Gamma_i\end{gather*}
is central in the homogeneous ring $\text{Dif\/f}_{\h,0}(n)$, see the calculation in \cite[Proposition~3]{HO}. Hence we have to track only those ordered terms whose f\/iltration degree, see~(\ref{defist}), is smaller than~3. As before, we use the symbol $u\big|_{\text{l.d.t.}}$ to denote these lower degree terms in an expression~$u$. We have
\begin{gather*}\big(x^j c(t)-c(t)x^j\big)\big|_{\text{l.d.t.}}=\left(-\frac{e(t)}{1+\th_j t}\sigma_j-\rho(t)[-\varepsilon_j]+\rho(t)\right) x^j.\end{gather*}
Thus the element $c(t)$ commutes with the generators $x^j$, $j=1,\dots,n$, if and only if the polynomial $\rho(t)$ satisf\/ies the system~(\ref{corho}).
The use of the anti-automorphism (\ref{antiautomb}) shows that the element $c(t)$ then commutes with the generators $\der_j$, $j=1,\dots,n$, as well.

 (ii) We check the case $j=1$. The calculation for $\sigma\in\mathcal{W}_j$ is similar.

Since the combination $\frac{e(t)}{1+\th_1 t}$ does not depend on $\th_1$, we have, for $\rho(t)=\frac{e(t)}{1+\th_1 t} \sigma$,
\begin{gather*}\Delta_1\rho(t)=\frac{e(t)}{1+\th_1 t}\Delta_1\sigma=\frac{e(t)}{1+\th_1 t}\sigma_1.\end{gather*}
For $j>1$ we have
\begin{gather}\label{deasi}\sigma=\big(\th_{1j}+1\big)\Delta_j\sigma,\qquad j=2,\dots,n,\end{gather}
and we calculate
\begin{gather*} \Delta_j\rho(t)=
\frac{e(t)}{\big(1+\th_1 t\big)\big(1+\th_j t\big)}\Delta_j\big(\big(1+\th_j t\big)\sigma\big) \\
\hphantom{\Delta_j\rho(t)}{}
= \frac{e(t)}{\big(1+\th_1 t\big)\big(1+\th_j t\big)}\big( t\sigma+\big(1+\big(\th_j-1\big)t\big)\Delta_j\sigma\big)= \frac{e(t)}{1+\th_j t}\Delta_j\sigma ,\end{gather*}
according to the formula (\ref{deasi}).

(iii) The proof is the same as for the ring $\Diff(n)$, see \cite[Lemma 8]{HO}.

\subsection{Rings of fractions. Proof of Lemma \ref{leis2}}\label{seisrf}

(i) The set $\mathfrak{B}_{\text{D}}:=\big\{\th_i,\Z'^{\circ i},c_i\big\}_{i=1}^n$, where $\Z'^{\circ i}:=\Z^i\psi_i'$, $i=1,\dots,n$, generates the localized ring $\text{S}_\Z^{-1}\Diffs (n)$. Moreover, the complete set of the def\/ining relations for the generators from the set $\mathfrak{B}_{\text{D}}$ does not remember about the potential $\sigma$. It reads
\begin{gather*}
\th_i\th_j=\th_j\th_i,\qquad \th_i \Z'^{\circ j}=\Z'^{\circ j}\big(\th_i+\delta_i^j\big),\qquad \Z'^{\circ i}\Z'^{\circ j}=\Z'^{\circ j}\Z'^{\circ i},\qquad i,j=1,\dots,n,\\
c_i\ \text{are central},\qquad i=1,\dots,n.
\end{gather*}
The proof is the same as for the ring $\Diff(n)$, see \cite{HO}. The isomorphism is now clear.

(ii) Assume that $\iota\colon \Diffs(n)\to\text{Dif\/f}_{\h,\sigma'}(n)$ is an isomorphism of f\/iltered rings over $\bar{\U}(n)$. To distinguish the generators, we denote the generators of the ring $\text{Dif\/f}_{\h,\sigma'}(n)$ by $\Z'^{i}$ and $\der'_i$.

The $\varepsilon_i$-weight subspace $\mathfrak{E}_i$ of the ring $\Diffs(n)$ consists of elements of the form $\theta \Z^i$ where~$\theta$ is a polynomial in the elements $\Gamma_j$, $j=1,\dots,n$, with coef\/f\/icients in $\bar{\U}(n)$. Since the space of the elements of $\mathfrak{E}_i$ of f\/iltration degree $\leq 1$ is $\bar{\U}(n)\Z^i$, we must have
\begin{gather}\label{priso}\iota\colon \ \Z^i\mapsto \mu_i\Z'^i,\qquad \der_i\mapsto\der'_i\nu_i\end{gather}
with some invertible elements $\mu_i,\nu_i\in \bar{\U}(n)$, $i=1,\dots,n$. Let $\gamma_i:=\mu_i\nu_i$, $i=1,\dots,n$. The def\/ining relation~(\ref{hdesp10c}) and the corresponding relation for the ring $\text{Dif\/f}_{\h,\sigma'}(n)$ shows that the formulas~(\ref{priso}) may def\/ine an isomorphism only if
\begin{gather}\label{noha1}\gamma_i=\gamma_j[\varepsilon_j],\qquad i,j=1,\dots,n,\end{gather}
and
\begin{gather}\label{noha2}\gamma_i\sigma_i'=\sigma_i,\qquad i=1,\dots,n.\end{gather}
The condition (\ref{noha1}) implies that $\gamma_i=\gamma$ for some $\gamma\in\mathbb{K}$. The condition~(\ref{noha2}) then becomes $\gamma\sigma_i'=\sigma_i$
and the assertion follows.

\subsection{Lowest weight representations. Proof of Proposition \ref{anosege}}\label{ansetv}
We need the following identity (see \cite[Lemma~5]{HO}):
\begin{gather}\label{neio}\sum_j\frac{1}{\th_j+t^{-1}}\QQ^+_j=1-\frac{e(t)[-\varepsilon]}{e(t)} \end{gather}
and its several consequences. At $t=\big(1-\th_m\big)^{-1}$, $m=1,\dots,n$, the equality (\ref{neio}) becomes
\begin{gather}\label{neio2}\sum_j \frac{1}{\th_{jm}+1} \QQ^+_j=1.\end{gather}
Then,
\begin{gather}
\sum_i \frac{1}{1+t\th_i} \frac{1}{\th_{ik}+1} \QQ^+_i = \frac{1}{1+t\big(\th_k-1\big)}\sum_i\left( \frac{1}{\th_{ik}+1}-\frac{t}{1+t\th_i}\right)\QQ^+_i \nonumber\\
\hphantom{\sum_i \frac{1}{1+t\th_i} \frac{1}{\th_{ik}+1} \QQ^+_i}{}
= \frac{1}{1+t\big(\th_k-1\big)}\frac{e(t)[-\varepsilon]}{e(t)}.\label{vspna1} \end{gather}
We used (\ref{neio}) and (\ref{neio2}) in the last equality. The substitution $\th_i\rightsquigarrow -\th_i+1$, $i=1,\dots,n$, and $t\rightsquigarrow -t$ into (\ref{vspna1}) gives
\begin{gather}\label{neio3}\sum_i\frac{1}{1+t\big(\th_i-1\big)}\frac{1}{\th_{ki}+1}\QQ^-_i=\frac{1}{1+t\th_k}\frac{e(t)}{e(t)[-\varepsilon]}.\end{gather}

{\bf Proof of Proposition \ref{anosege}.} Since the element $c(t)$ is central, it is suf\/f\/icient to calculate its value on the vector $\vert\;\rangle$. Denote
\begin{gather*} c(t)\vert\;\rangle=\omega(t)\vert\;\rangle.\end{gather*}

We have
\begin{gather}\label{inredx}\der_j x^i=\sum_{k,l}\PPsi_{jl}^{ik}x^l \der_k +\sum_k \PPsi_{jk}^{ik}\sigma_k,\end{gather}
where $\PPsi$ is the skew inverse of the operator~$\RR$, see~(\ref{skewn}) (we refer, e.g., to \cite[Section 4.1.2]{O} for details on skew inverses).

Since the generators $\der_i$, $i=1,\dots,n$, annihilate the vector $\vert\;\rangle$, see~(\ref{dmoac}), we f\/ind, in view of~(\ref{inredx}), that
\begin{gather} \omega(t)=
 \sum_{i,k}\frac{e(t)}{1+t\th_i}\PPsi^{ik}_{ik}\sigma_k -\rho(t)=\sum_{i,k}\frac{e(t)}{1+t\th_i} \frac{1}{\th_{ik}+1}\QQ^+_i \QQ^-_k\sigma_k -\rho(t)\nonumber\\
\hphantom{\omega(t)}{}
= e(t)[-\varepsilon]\underline{ \sum_k \frac{1}{1+t(\th_k-1)} \QQ^-_k\sigma_k } -\rho(t). \label{nuidoal1} \end{gather}
We used (\ref{explpsi}) in the second equality and (\ref{vspna1}) in the third equality.

We shall verify (\ref{reoacec}) for every representative of the space $\mathcal{W}$. As in the proof of Proposi\-tion~\ref{lece}(ii), it is suf\/f\/icient to establish~(\ref{reoacec}) for
\begin{gather*}\sigma=\frac{A\big(\th_1\big)}{\chi_1},\qquad \text{where}\ A\ \text{is a univariate polynomial}.\end{gather*}
Then
\begin{gather}\label{neio5}
\sigma_j=\frac{1}{\th_{1j}+1}\sigma,\qquad j=2,\dots,n,\end{gather}
and, according to Proposition \ref{lece}(ii),
\begin{gather}\label{neio6} \rho(t)=\frac{e(t)}{1+t\th_1}\sigma.\end{gather}
Denote the underlined sum in (\ref{nuidoal1}) by $\xi$. Taking into account (\ref{neio5}) we calculate
\begin{gather*} \xi =
\frac{1}{1+t\big(\th_1-1\big)} (\sigma-\sigma[-\varepsilon_1] )\QQ^-_1
+\sigma\sum_{j=2}^n \frac{1}{1+t\big(\th_j-1\big)} \frac{1}{\th_{1j}+1}\QQ^-_j \\
\hphantom{\xi}{}
 = -\frac{\sigma[-\varepsilon_1]\QQ^-_1}{1+t\big(\th_1-1\big)}+\sigma\sum_{j=1}^n \frac{1}{1+t\big(\th_j-1\big)} \frac{1}{\th_{1j}+1}\QQ^-_j\\
\hphantom{\xi}{}
= -\frac{\sigma[-\varepsilon_1]\QQ^-_1}{1+t\big(\th_1-1\big)}+\frac{1}{1+t\th_1}\frac{\sigma e(t)}{e(t)[-\varepsilon]}.\end{gather*}
We have used (\ref{neio3}) in the last equality. Note that
\begin{gather*}\sigma[-\varepsilon_1] \QQ^-_1=\frac{A\big(\th_1-1\big)}{\chi_1[-\varepsilon_1]} \frac{\chi_1[-\varepsilon_1]}{\chi_1}=\frac{A\big(\th_1-1\big)}{\chi_1}=\sigma[-\varepsilon],\end{gather*}
so
\begin{gather*}\xi=-\frac{\sigma[-\varepsilon]}{1+t\big(\th_1-1\big)}+\frac{1}{1+t\th_1}\frac{\sigma e(t)}{e(t)[-\varepsilon]}.\end{gather*}
Substituting the obtained expression for $\xi$ into (\ref{nuidoal1}) and taking into account (\ref{neio6}) we conclude that
\begin{gather*}
\omega(t)= e(t)[-\varepsilon]\left( -\frac{\sigma[-\varepsilon]}{1+t\big(\th_1-1\big)}+
\frac{1}{1+t\th_1}\frac{\sigma e(t)}{e(t)[-\varepsilon]}\right)-\frac{e(t)}{1+t\th_1}\sigma \\
\hphantom{\omega(t)}{}=- \frac{e(t)[-\varepsilon]}{1+t\big(\th_1-1\big)}\sigma[-\varepsilon] =-\rho(t)[-\varepsilon],\end{gather*}
as stated.

\subsection{Several copies. Proof of Lemma \ref{seco}}\label{seseco}

Assume that, say, $N>1$. Repeating the calculations (\ref{fiwa}) and (\ref{sewa}) for one copy in Section~\ref{PBWle1}, we f\/ind, for $i,j,k=1,\dots,n$,
\begin{gather}
\big( x^{i\alpha}\der_{j\beta}\big) \der_{k\gamma}\big|_{\text{l.d.t.}}=
\left( \sum_{u,v}\RR^{ui}_{vj}[\varepsilon_u]\der_{u\beta}x^{v\alpha}-\delta^i_j\sigma_{i\alpha\beta}\right)\der_{k\gamma}
\big|_{\text{l.d.t.}} \nonumber\\
\hphantom{\big( x^{i\alpha}\der_{j\beta}\big) \der_{k\gamma}\big|_{\text{l.d.t.}}}{}
 = -\sum_{u}\RR^{ui}_{kj}[\varepsilon_u]\der_{u\beta}\sigma_{k\alpha\gamma}-\delta^i_j\sigma_{i\alpha\beta}\der_{k\gamma},\label{fiwanb1}
\\
x^{i\alpha}\big( \der_{j\beta} \der_{k\gamma}\big)\big|_{\text{l.d.t.}}= x^{i\alpha}\sum_{a,b}\RR^{ab}_{kj}\der_{b\gamma}\der_{a\beta}\big|_{\text{l.d.t.}} \nonumber\\
\hphantom{x^{i\alpha}\big( \der_{j\beta} \der_{k\gamma}\big)\big|_{\text{l.d.t.}}}{} =
\sum_{a,b}\RR^{ab}_{kj}[-\varepsilon_i]\left( \sum_{c,d}\RR^{ci}_{db}[\varepsilon_c]\der_{c\gamma}x^{d\alpha}-\delta^i_b\sigma_{i\alpha\gamma}\right)\der_{a\beta}
\big|_{\text{l.d.t.}} \nonumber\\
\hphantom{x^{i\alpha}\big( \der_{j\beta} \der_{k\gamma}\big)\big|_{\text{l.d.t.}}}{}=
-\sum_{a,b,c}\RR^{ab}_{kj}[-\varepsilon_i] \RR^{ci}_{ab}[\varepsilon_c]\der_{c\gamma}\sigma_{a\alpha\beta}-\sum_{a}\RR^{ai}_{kj}[-\varepsilon_i]
\sigma_{i\alpha\gamma}\der_{a\beta} . \label{sewanb1}\end{gather}
Take $\beta\neq\gamma$.
Equating the coef\/f\/icients in $\der_{u\beta}$, $u=1,\dots,n$, in (\ref{fiwanb1}) and (\ref{sewanb1}), we f\/ind
\begin{gather}\label{ysy1}\RR^{ui}_{kj}[\varepsilon_u]\sigma_{k\alpha\gamma}[\varepsilon_u]=\RR^{ui}_{kj}[-\varepsilon_i]\sigma_{i\alpha\gamma}
,\qquad i,k,j,u=1,\dots,n.\end{gather}
Equating the coef\/f\/icients in $\der_{u\gamma}$, $u=1,\dots,n$, in (\ref{fiwanb1}) and (\ref{sewanb1}), we f\/ind
\begin{gather}\label{ysy2}\delta^i_j\delta^u_k\sigma_{i\alpha\beta}=\sum_{a,b}\RR^{ab}_{kj}[-\varepsilon_i] \RR^{ui}_{ab}[\varepsilon_u]
\sigma_{a\alpha\beta}[\varepsilon_u],\qquad i,k,j,u=1,\dots,n.\end{gather}

Shifting by $-\varepsilon_u$ and using the property (\ref{weze}) we rewrite the equality (\ref{ysy1}) in the form
\begin{gather}\label{kosha1}\RR^{ui}_{kj}\left(\sigma_{k\alpha\gamma}-\sigma_{i\alpha\gamma}[-\varepsilon_u]\right)=0.\end{gather}
Setting $u=k$ and $j=i$ (with arbitrary $i,k=1,\dots,n$) in (\ref{kosha1}), we obtain
\begin{gather*}\sigma_{k\alpha\gamma}=\sigma_{i\alpha\gamma}[-\varepsilon_k], \end{gather*}
which implies the assertion (\ref{posinb1}).

A direct calculation, with the help of the properties (\ref{weze}), (\ref{iceco}) and (\ref{Q5do}) of the operator~$\RR$, shows that the condition~(\ref{posinb1}) implies the equalities (\ref{ysy1}) and (\ref{ysy2}) as well as all the remaining conditions for the f\/latness of the deformation.

\subsection*{Acknowledgements}

The work of O.O.\ was supported by the Program of Competitive Growth of Kazan Federal University and by the grant RFBR 17-01-00585.

\pdfbookmark[1]{References}{ref}
\LastPageEnding


\begin{thebibliography}{99}
\footnotesize\itemsep=0pt

\bibitem{AF}
Alekseev A.Y., Faddeev L.D., {$(T^*G)_t$}: a toy model for conformal f\/ield
 theory, \href{https://doi.org/10.1007/BF02101512}{\textit{Comm. Math. Phys.}} \textbf{141} (1991), 413--422.

\bibitem{B}
Bergman G.M., The diamond lemma for ring theory, \href{https://doi.org/10.1016/0001-8708(78)90010-5}{\textit{Adv. Math.}}
 \textbf{29} (1978), 178--218.

\bibitem{Bo}
Bokut' L.A., Embeddings into simple associative algebras, \href{https://doi.org/10.1007/BF01877233}{\textit{Algebra
 Logic}} \textbf{15} (1976), 73--90.

\bibitem{BF}
Bytsko A.G., Faddeev L.D., {$(T^*{\mathcal B})_q$}, {$q$}-analog of model space
 and the {C}lebsch--{G}ordan coef\/f\/icients generating matrices,
 \href{https://doi.org/10.1063/1.531780}{\textit{J.~Math. Phys.}} \textbf{37} (1996), 6324--6348,
 \href{https://arxiv.org/abs/q-alg/9508022}{q-alg/9508022}.

\bibitem{FHIOPT}
Furlan P., Hadjiivanov L.K., Isaev A.P., Ogievetsky O.V., Pyatov P.N., Todorov
 I.T., Quantum matrix algebra for the {${\rm SU}(n)$} {WZNW} model,
 \href{https://doi.org/10.1088/0305-4470/36/20/310}{\textit{J.~Phys.~A: Math. Gen.}} \textbf{36} (2003), 5497--5530,
 \href{https://arxiv.org/abs/hep-th/0003210}{hep-th/0003210}.

\bibitem{GT}
Gel'fand I.M., Tsetlin M.L., Finite-dimensional representations of the group of
 unimodular matrices, \textit{Dokl. Akad. Nauk USSR} \textbf{71} (1950),
 825--828, {E}nglish translation in Gelfand~I.M., Collected papers, Vol.~II,
 Springer-Verlag, Berlin, 1988, 653--656.

\bibitem{HIOPT}
Hadjiivanov L.K., Isaev A.P., Ogievetsky O.V., Pyatov P.N., Todorov I.T., Hecke
 algebraic properties of dynamical {$R$}-matrices. {A}pplication to related
 quantum matrix algebras, \href{https://doi.org/10.1063/1.532779}{\textit{J.~Math. Phys.}} \textbf{40} (1999),
 427--448, \href{https://arxiv.org/abs/q-alg/9712026}{q-alg/9712026}.

\bibitem{KN}
Khoroshkin S., Nazarov M., Mickelsson algebras and representations of
 {Y}angians, \href{https://doi.org/10.1090/S0002-9947-2011-05367-5}{\textit{Trans. Amer. Math. Soc.}} \textbf{364} (2012), 1293--1367,
 \href{https://arxiv.org/abs/0912.1101}{arXiv:0912.1101}.

\bibitem{KO1}
Khoroshkin S., Ogievetsky O., Mickelsson algebras and {Z}helobenko operators,
 \href{https://doi.org/10.1016/j.jalgebra.2007.04.020}{\textit{J.~Algebra}} \textbf{319} (2008), 2113--2165,
 \href{https://arxiv.org/abs/math.QA/0606259}{math.QA/0606259}.

\bibitem{KO2}
Khoroshkin S., Ogievetsky O., Diagonal reduction algebras of ${\bf gl}$ type,
 \href{https://doi.org/10.1007/s10688-010-0023-0}{\textit{Funct. Anal. Appl.}} \textbf{44} (2010), 182--198, \href{https://arxiv.org/abs/0912.4055}{arXiv:0912.4055}.

\bibitem{KO3}
Khoroshkin S., Ogievetsky O., Structure constants of diagonal reduction
 algebras of {${\bf gl}$} type, \href{https://doi.org/10.3842/SIGMA.2011.064}{\textit{SIGMA}} \textbf{7} (2011), 064,
 34~pages, \href{https://arxiv.org/abs/1101.2647}{arXiv:1101.2647}.

\bibitem{KO4}
Khoroshkin S., Ogievetsky O., Rings of fractions of reduction algebras,
 \href{https://doi.org/10.1007/s10468-012-9397-4}{\textit{Algebr. Represent. Theory}} \textbf{17} (2014), 265--274.

\bibitem{KO5}
Khoroshkin S., Ogievetsky O., Diagonal reduction algebra and the ref\/lection
 equation, \href{https://doi.org/10.1007/s11856-017-1571-2}{\textit{Israel~J. Math.}} \textbf{221} (2017), 705--729,
 \href{https://arxiv.org/abs/1510.05258}{arXiv:1510.05258}.

\bibitem{M}
Mickelsson J., Step algebras of semi-simple subalgebras of {L}ie algebras,
 \href{https://doi.org/10.1016/0034-4877(73)90006-2}{\textit{Rep. Math. Phys.}} \textbf{4} (1973), 307--318.

\bibitem{O}
Ogievetsky O., Uses of quantum spaces, in Quantum Symmetries in Theoretical
 Physics and Mathematics ({B}ariloche, 2000), \href{https://doi.org/10.1090/conm/294/04973}{\textit{Contemp. Math.}}, Vol.~294, Amer. Math. Soc., Providence, RI, 2002, 161--232.

\bibitem{HO}
Ogievetsky O., Herlemont B., Rings of {$\bf h$}-deformed dif\/ferential
 operators, \href{https://doi.org/10.1134/S0040577917080104}{\textit{Theoret. and Math. Phys.}} \textbf{192} (2017), 1218--1229,
 \href{https://arxiv.org/abs/1612.08001}{arXiv:1612.08001}.

\bibitem{T}
Tolstoy V.N., Fortieth anniversary of extremal projector method for {L}ie
 symmetries, in Noncommutative Geometry and Representation Theory in
 Mathematical Physics, \href{https://doi.org/10.1090/conm/391/07342}{\textit{Contemp. Math.}}, Vol.~391, Amer. Math. Soc.,
 Providence, RI, 2005, 371--384, \href{https://arxiv.org/abs/math-ph/0412087}{math-ph/0412087}.

\bibitem{H}
van~den Hombergh A., {H}arish-{C}handra modules and representations of step
 algebra, Ph.D.~Thesis, Katolic University of Nijmegen, 1976, available at
 \url{http://hdl.handle.net/2066/147527}.

\bibitem{WZ}
Wess J., Zumino B., Covariant dif\/ferential calculus on the quantum hyperplane,
 \href{https://doi.org/10.1016/0920-5632(91)90143-3}{\textit{Nuclear Phys.~B Proc. Suppl.}} \textbf{18} (1990), 302--312.

\bibitem{Zh1}
Zhelobenko D.P., Classical groups. {S}pectral analysis of f\/inite-dimensional
 representations, \href{https://doi.org/10.1070/RM1962v017n01ABEH001123}{\textit{Russian Math. Surveys}} \textbf{17} (1962), no.~1,
 1--94.

\bibitem{Zh2}
Zhelobenko D.P., Extremal cocycles on {W}eyl groups, \href{https://doi.org/10.1007/BF02577133}{\textit{Funct. Anal.
 Appl.}} \textbf{21} (1987), 183--192.

\bibitem{Zh}
Zhelobenko D.P., Representations of reductive Lie algebras, Nauka, Moscow,
 1994.

\end{thebibliography}
\end{document}